\documentclass[12pt]{article}
\usepackage[T1]{fontenc}
\usepackage[utf8]{inputenc}

\usepackage[american]{babel}
\usepackage[babel, final]{microtype}
\usepackage[all]{xy}
\usepackage{ifpdf}
\usepackage{comment}
\usepackage{multirow}

\usepackage[pdftex, dvipsnames]{xcolor}
\usepackage[pdftex, final]{graphicx}
\usepackage{pinlabel}
\usepackage[pdftex,%
    nomarginpar,%
    lmargin=1in,%
    rmargin=1in,%
    tmargin=1in,%
    bmargin=1in,%
]{geometry}
\usepackage{amsmath,amsthm,amssymb,amsfonts,mathrsfs,tikz,tikz-cd,bm,verbatim,mathtools,hyperref,caption,enumitem,csquotes,float,soul}
\usepackage[mathscr]{euscript}
\usepackage{thmtools, thm-restate}

\usepackage[margin=1cm]{caption}
\usepackage{blkarray}

\hypersetup{
    colorlinks=true,
    linkcolor=NavyBlue,
    citecolor=NavyBlue
}

\setlist{noitemsep} 

\usepackage[backend=biber,style=alphabetic,sorting=nyt, url=false,doi=false,isbn=false,backref=true,maxnames=10]{biblatex}
\DefineBibliographyStrings{english}{
  backrefpage={$\uparrow$},
  backrefpages={$\uparrow$}
}
\addto\extrasamerican{%
}

\newtheorem{corollary}{Corollary}[section]
\newtheorem{proposition}{Proposition}[section]
\newtheorem{lemma}{Lemma}[section]
\newtheorem{fact}{Fact}[section]

\theoremstyle{definition}
\newtheorem{definition}{Definition}[section]
\newtheorem{example}{Example}[section]
\newtheorem{remark}{Remark}[section]

\makeatletter
\let\c@conjecture=\c@theorem
\let\c@corollary=\c@theorem
\let\c@proposition=\c@theorem
\let\c@lemma=\c@theorem
\let\c@definition=\c@theorem
\let\c@example=\c@theorem
\let\c@remark=\c@theorem
\let\c@equation\c@theorem
\let\c@question\c@theorem
\let\c@fact\c@theorem
\let\c@observation\c@theorem
\makeatother

\def\makeautorefname#1#2{\expandafter\def\csname#1autorefname\endcsname{#2}}
\makeautorefname{proposition}{Proposition}
\makeautorefname{lemma}{Lemma}
\makeautorefname{definition}{Definition}
\makeautorefname{example}{Example}
\makeautorefname{question}{Question}
\makeautorefname{conjecture}{Conjecture}
\makeautorefname{section}{Section}
\makeautorefname{claim}{Claim}
\makeautorefname{equation}{Equation}
\makeautorefname{remark}{Remark}
\makeautorefname{corollary}{Corollary}
\makeautorefname{fact}{Fact}
\makeautorefname{observation}{Observation}

\makeatletter
\def\@tocline#1#2#3#4#5#6#7{\relax
  \ifnum #1>\c@tocdepth 
  \else
    \par \addpenalty\@secpenalty\addvspace{#2}%
    \begingroup \hyphenpenalty\@M
    \@ifempty{#4}{%
      \@tempdima\csname r@tocindent\number#1\endcsname\relax
    }{%
      \@tempdima#4\relax
    }%
    \parindent\z@ \leftskip#3\relax \advance\leftskip\@tempdima\relax
    \rightskip\@pnumwidth plus4em \parfillskip-\@pnumwidth
    #5\leavevmode\hskip-\@tempdima
      \ifcase #1
       \or\or \hskip 1em \or \hskip 2em \else \hskip 3em \fi%
      #6\nobreak\relax
    \dotfill\hbox to\@pnumwidth{\@tocpagenum{#7}}\par
    \nobreak
    \endgroup
  \fi}
\makeatother

\title{Bridge indices of spatial graphs and diagram colorings}
\author{Sarah Blackwell, Puttipong Pongtanapaisan, and Hanh Vo}
\date{}

\begin{document}
\maketitle
\begin{abstract}
We extend the Wirtinger number of links, an invariant originally defined by Blair, Kjuchukova, Velazquez, and Villanueva in terms of extending initial colorings of some strands of a diagram to the entire diagram, to spatial graphs. We prove that the Wirtinger number equals the bridge index of spatial graphs, and we implement an algorithm in Python which gives a more efficient way to estimate upper bounds of bridge indices. Combined with lower bounds from diagram colorings by elements from certain algebraic structures and clasping techniques, we obtain exact bridge indices for a large family of almost unknotted spatial graphs. We also show that for every possible negative Euler characteristic, there exist almost unknotted graphs of arbitrarily large bridge index. 
\end{abstract}


\section{Introduction}

The \textit{bridge index} has been a useful complexity measure of knots and links since the 1950s, when Schubert used it to show that a knot has only finitely many companions \cite{schubert1954numerische}. The bridge index has also been used to estimate several geometric quantities such as distortion \cite{pardon2011distortion,blair2020distortion}, lattice stick number \cite{adams2012stick}, and total curvature \cite{milnor1950total}. It is conjectured that the bridge index equals an algebraic quantity called the \textit{meridional rank}; a counterexample has not yet been discovered (see Problem 1.11 in Kirby's list \cite{kirby1997problems}). Baader, Blair, and Kjuchukova  verified that this conjecture holds for a large family of links by rephrasing the bridge index algorithmically \cite{baader2021coxeter}, building off of work by Blair, Kjuchukova, Velazquez, and Villanueva \cite{blair2020wirtinger}.

Knots and links can be considered as a proper subclass of \textit{spatial graphs}, which encompass embeddings of 1-complexes. These graphs embedded in 3-space are objects of interest to chemists, as atoms of a molecule and their chemical bonds can be modeled as vertices and edges of a spatial graph, respectively \cite{castle2008ravels}. Additionally, conclusions about abstract graphs can be made by considering how they are spatially embedded in space. For example, the complete graphs $K_6$ and $K_7$ fall into two distinct levels in a filtration of abstract graphs formed by considering the minimum number of links in any embedding of the graph: any embedding of $K_6$ (resp. $K_7$) in 3-space contains at least one link (resp. 21 links) \cite{fleming2009counting}.

Various authors have generalized the bridge index to spatial graphs, and some have shown that their version still approximates curvature as the bridge index of knots does \cite{soma2002curvature,gulliver2012total}. There are some properties however that do not carry over to the more general setting. For instance, any knot with bridge index exactly two is prime \cite{schubert1956knoten}, but this is no longer true for $\Theta$-graphs \cite{motohashi20002}. For non-prime spatial graphs, there exist estimates for the bridge indices in terms of the number of summands \cite{taylor2021tunnel}.

Since a random spatial graph encountered in nature usually does not have a simple diagram, it is beneficial to generalize the method of \cite{blair2020wirtinger} to systematically compute the bridge indices of spatial graphs from inputs that are computer-friendly. Our procedure can be summarized as follows. Given a spatial graph diagram, we pick a subset of ``strands'' to color and define a move that extends a color to an additional strand provided some conditions are met at a crossing. Roughly, the \textit{Wirtinger number} is (a weighted count of) the minimum number of initial colorings that extend to the entire diagram by the move, minimized over all diagrams (see \autoref{sec:wirtdef} for a formal definition). This quantity turns out to equal the bridge index, which we prove in \autoref{sec:result}. 

\begin{restatable}{theorem}{main}
\label{thm:main}%
    For spatial graphs, the Wirtinger number equals the bridge index.
\end{restatable}

We use our main result, which we implement as an algorithm in Python, to provide upper estimates for the bridge index of a variety of spatial graphs. In some cases we combine our techniques with algebraic lower bounds, which we detail in \autoref{sec:bounds}, to compute exact values. We use our methods to prove in \autoref{sec:large} that the bridge index of almost unknotted spatial graphs can be arbitrarily large. We accomplish this by modifying clasping and summing operations to be compatible with quandle colorings. Note that a connected graph $G$ with Euler characteristic $\chi(G)=\beta_0(G)-\beta_1(G)\geq 0$ is either a link or is not topologically interesting.

\begin{restatable}{theorem}{index}
\label{thm:index}%
   For every possible negative integer $m$, there exist almost
unknotted graphs $G$ of arbitrarily large bridge index with Euler characteristic $\chi(G)=m$.
\end{restatable}

Throughout \autoref{sec:examples} we provide many other example computations. In particular, spatial graph whose vertices all have degree four can be obtained from link diagrams by turning some crossings into vertices, or fusing maxima and minima. We can compute exact bridge indices for spatial graphs obtained from generalized Montesinos links by fusing some maxima and minima together (see \autoref{ex:Montesinos}). In some cases, these fusing operations can still produce interesting almost unknotted graphs. 

Our Python code can be found at 
\href{https://github.com/hanhv/graph-wirt}{https://github.com/hanhv/graph-wirt}.

\subsection*{Acknowledgments} Some of this work was conducted when the first author visited the other two at Arizona State University, and she thanks the mathematics department for their hospitality. We would like to thank Julien Paupert for support and helpful advice. We also thank the reviewer for many helpful suggestions that improved the paper. SB was supported by the NSF Postdoctoral Research Fellowship DMS-2303143.

\setcounter{tocdepth}{2}
\tableofcontents

\section{Preliminaries} \label{sec:prelim}
\subsection{Spatial graphs} \label{sec:graphs}
A \textit{graph} $G$ is a pair $(V,E)$ of a finite, non-empty set $V$ of vertices and a finite multiset $E$ of edges, where $G = E \subset V \times V$. A \textit{spatial graph} is an embedding of a graph $G$ in $S^3$ taking vertices to points in $S^3$ and taking edges $(u,v)$ to an arc whose endpoints are images of $u$ and $v$. We assume that our spatial graphs are ``tame'' in the sense of \cite{friedl2021graphical}. We say a spatial graph is \textit{trivial} if it can be isotoped to lie in the plane. 

A vertex $v$ has \textit{degree} $k$, or is \textit{$k$-valent}, if it is connected to $k$ edges. In this paper, we do not consider embeddings of graphs with degree one vertices. Topologically, we can isotope the edge adjacent to a degree one vertex freely so that it does not contribute to the knotting. We treat degree two vertices as parts of edges; that is, we consider the two edges connected to the vertex to be the same as if we combined the two edges into one edge without a vertex there.

Two of the most simple types of spatial graphs, which we will encounter throughout the paper, are \textit{$\Theta$-graphs} and \textit{handcuff graphs}. A spatial $\Theta$-graph has two vertices $u$ and $v$, with the underlying graph type $\{(u,v),(u,v),(u,v)\}$. A spatial handcuff graph also has two vertices, but with the underlying graph type $\{(u,u),(u,v),(v,v)\}$. We can generalize the notion of $\Theta$-graphs to \textit{$\Theta_n$-graphs}, which have $n$ edges between the two vertices (for $n\geq 4$).

A straightforward way to construct a spatial graph is to start with a link, and then add some edges. This way, the complexity of the knotted graph is obtained from the original link. This motivates the study of the following class of spatial graphs. A nontrivial spatial graph is \textit{almost unknotted}\footnote{Sometimes this is also referred to as \textit{minimally knotted} or \textit{Brunnian} in the literature.} if all of its proper subgraphs are trivial. For $\Theta$-graphs, this means each \textit{constituent knot}, that is, the knot obtained by deleting one of the three edges, is the unknot. 

A common way to construct more complicated spatial graphs is to form composite spatial graphs, analogous to forming composite numbers from prime numbers. Let $G_1,G_2$ be spatial graphs in $S^3.$ Consider a point $p_1\in G_1$ and another point $p_2\in G_2$. Removing small regular neighborhoods of $p_1,p_2$ create spherical boundaries that can be identified together. The resulting space is another spatial graph in $S^3$. If $p_1$ and $p_2$ both lie in the interior of edges, then the resulting graph is a \textit{connected sum} and we denote it by $G_1\# G_2$. When $p_1$ and $p_2$ are both vertices of degree $k$, then the resulting graph is a \textit{$k$-valent vertex sum} and we denote it by $G_1\#_k G_2$.

\subsection{Bridge splittings}
There are several definitions for the bridge index of spatial graphs, but the following formulation, which also appeared in the work of Ozawa \cite{ozawa2012bridge}, is compatible with widely used concepts such as thin positions and tunnel systems. We call a properly embedded graph $\tau$ in a 3-ball $B$ a \textit{trivial tangle} if components of $\tau$ are trees or arcs and there exists a disk $D$ in $B$ containing $\tau$; see \autoref{fig:trees}.

\begin{figure}[ht!]
\includegraphics[width=10cm]{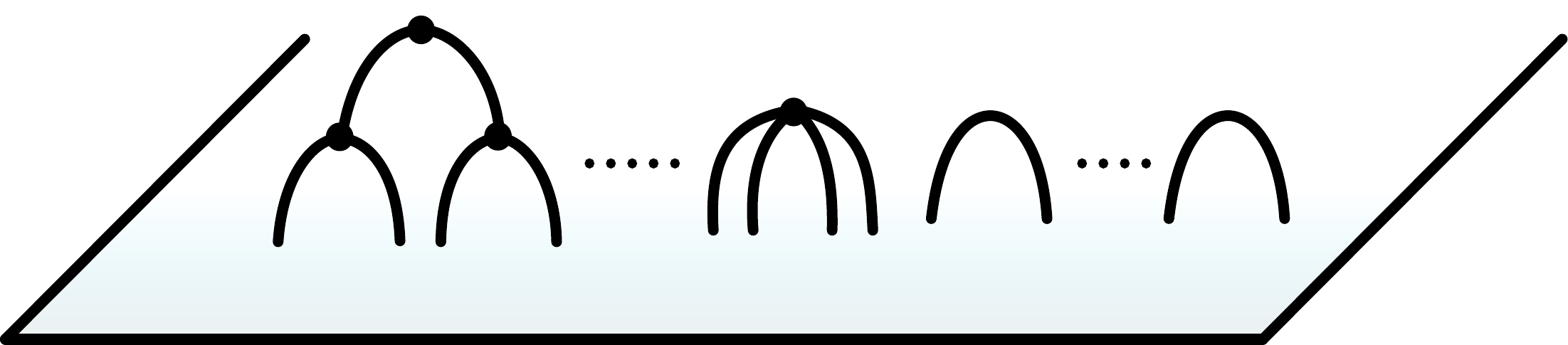}
\centering
\caption{A trivial tangle, which contains some number of trees and arcs (not necessarily grouped together as shown). We did not include vertices where the tangle intersects $\partial B$, since when we glue two trivial tangles together along these vertices, we think of the resulting closed spatial graph as no longer having vertices where we glued. \label{fig:trees}}
\end{figure}

A \textit{bridge splitting} of a spatial graph $G$ in the $3$-sphere $S^3$ is a decomposition into two trivial tangles $(S^3,G) = (B_1,\tau_1)\cup_{S^2} (B_2,\tau_2)$. The splitting sphere $S^2$ is also called the \textit{bridge sphere}.
For bridge splittings of links in $S^3$, each tangle has the same number of maxima with respect to the radial height function on the ball $B_i$. We call the minimum number of maxima over all possible bridge splittings the \textit{bridge index}. Note however that in general the two tangles in a bridge splitting of a spatial graph may have a different number of maxima with respect to the radial height function on the ball $B_i$, where the graph $\tau_i$
is isotoped so that each component contains a single maximum. For this reason there exist multiple definitions of bridge index for spatial graphs in the literature. Our definition will be the following.

\begin{definition} \label{def:bridge}
    The \textit{bridge index of a particular bridge splitting $(S^3,G) = (B_1,\tau_1)\cup_{S^2} (B_2,\tau_2)$ of a spatial graph $G$}, denoted $\beta(S^3,G)$, is defined to be $\tfrac{1}{2}|S^2\cap G|$. The \textit{bridge index of a spatial graph $G$}, denoted $\beta(G)$, is the minimum bridge index over all bridge splittings for $G$. 
\end{definition}

\noindent When it is clear from context, we may abuse notation and use $\beta(G)$ to denote the bridge index of a particular bridge splitting of a graph as well.

Another method for defining the bridge index for spatial graphs which has appeared previously in the literature is to count the number of components in a trivial tangle (and then minimize this quantity over all possible bridge splittings). Historically this method has been used in the context of studying $\Theta_n$-graphs \cite{goda1997bridge,motohashi20002,tsuno2003bridge}, with the constraint that each tangle contains only one vertex of $G$ so that the number of components of the two tangles is the same. In general, this may not always be the case, but one possible fix is to take the minimum number of components between the two tangles. 

Once the number of tangle components is known, along with the values of the weights of each component (see \autoref{def:strand}), we can recover the number of intersection points between the spatial graph embedding and the bridge sphere, which is the sum of the weights of the components of a trivial tangle. That is, if  $\min(|\tau_1|,|\tau_2|) = i$, then we have that
\begin{align*}
       \beta(S^3,G) = \tfrac{1}{2} |S^2\cap G| = \tfrac{1}{2} \sum_i w_i,
\end{align*}
where the $w_i$ are the weights of the $i$ components.

Furthermore, once the number of components of one of the trivial tangles $|\tau_1|$ is known, one can calculate the number of components of the other trivial tangle $|\tau_2|$ in a bridge splitting as follows. Let $\chi(G)$ be the number of vertices of $G$ minus the number of edges of $G$. To keep track of the local maxima, we put a vertex of degree two at each local maximum which is a bridge arc. At the moment, the number of vertices is $|\tau_1|+|\tau_2|.$ The number of edges can then be determined by examining the sum of the degrees of the vertices in $\tau_1$: $\sum_{1\leq i\leq |\tau_1|} d_i$. It follows that 
   \begin{align*}
       \chi(G)=|\tau_1|+|\tau_2|-\sum_{1\leq i\leq |\tau_1|} d_i.
   \end{align*}

\subsection{Wirtinger number} \label{sec:wirtdef}
Given a spatial graph $G$ in $S^3$, a \textit{diagram} $D$ for $G$ is a generic immersion of $G$ into the plane with each crossing labeled as an overcrossing or an undercrossing. 
An \textit{arc} of a spatial graph diagram is a subset of the diagram which starts at an undercrossing or vertex and traces the diagram until it meets another undercrossing or vertex, at which point it ends. 

For the following definition, it is useful to recast a diagram as a (possibly) disconnected planar graph, where degree one vertices are placed on either side of an undercrossing and the piece of the diagram in between is deleted, as in \autoref{fig:planar}. For the sake of clarity, we will refer to these degree one vertices exclusively as \textit{endpoints} (or \textit{free ends}); that is, ``vertex'' will always imply degree three or more, as assumed in \autoref{sec:graphs}. In other words, vertices of the diagram correspond to vertices of the spatial graph, while endpoints are artifacts of the diagram which we have added.

\begin{figure}[ht!]
\includegraphics[width=2.75cm]{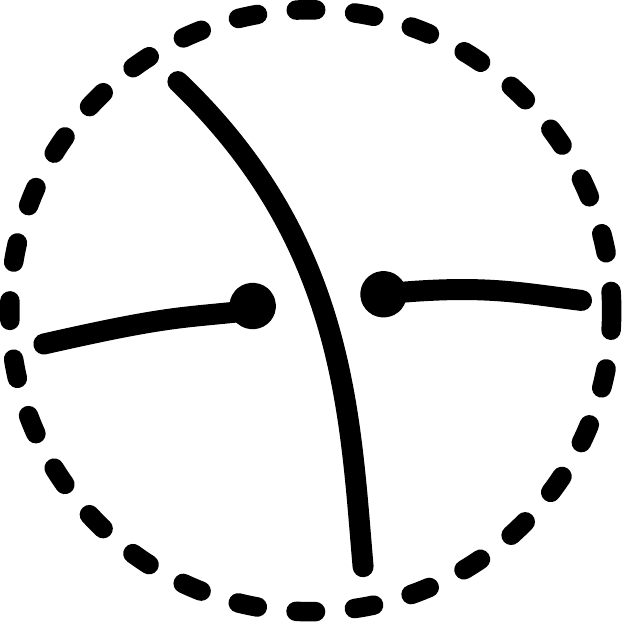}
\centering
\caption{When we recast a spatial graph diagram as a planar graph, we delete small neighborhoods of the undercrossings and add endpoints to each side of the deleted portion. \label{fig:planar}}
\end{figure}

\begin{definition} \label{def:strand}
A \textit{strand} of a diagram for a spatial graph is a connected component of the diagram, when thought of as a planar graph as described above.
The \textit{weight} of a strand is the number of endpoints of the strand. 
\end{definition}

\noindent Arcs which both begin and end at an undercrossing, such as the green arc in \autoref{fig:strands}, are strands with weight $2$. Arcs which begin or end at vertices, such as the red arc in \autoref{fig:strands}, are part of a larger strand, but not strands themselves. Strands with only one vertex, such as the blue strand in \autoref{fig:strands}, have weight equal to the degree of the vertex.

\begin{figure}[ht!]
\includegraphics[width=7cm]{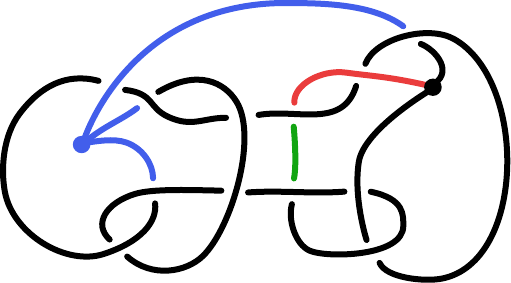}
\centering
\caption{Some examples (blue and green) and non-examples (red) of strands. The blue strand has weight $3$ and the green strand has weight $2$. \label{fig:strands}}
\end{figure}

We now describe a process for coloring a spatial graph diagram, analogous to the procedure detailed in \cite{blair2020wirtinger} for links. First choose $k$ distinct colors, and color a subset of strands of $D$ using these $k$ colors, which we call \textit{seeds}.
Sometimes we will also refer to the strands which are colored the initial colors themselves as seeds.
Then continue to color any uncolored arcs using the following \textit{coloring move}, which extends the coloring of a colored arc which passes under another colored arc to the other arc passing under the overstrand. See \autoref{fig:moves1}. In particular, once a vertex is reached, the coloring does not extend across the vertex; the other arcs connected to the same vertex must receive their color from a coloring move elsewhere in the diagram. 

Note that the coloring move is valid both when the overstrand matches and does not match the color of the understrand; all that matters is that the overstrand has been colored. Also, while strands which are included as seeds will contain multiple arcs all of the same color (unless the strand itself is just one arc), in general arcs connected to the same vertex can end up being colored with different colors if not included as a seed. See \autoref{fig:sequence1} for an example of this phenomenon.

\begin{figure}[ht!]
\includegraphics[width=8cm]{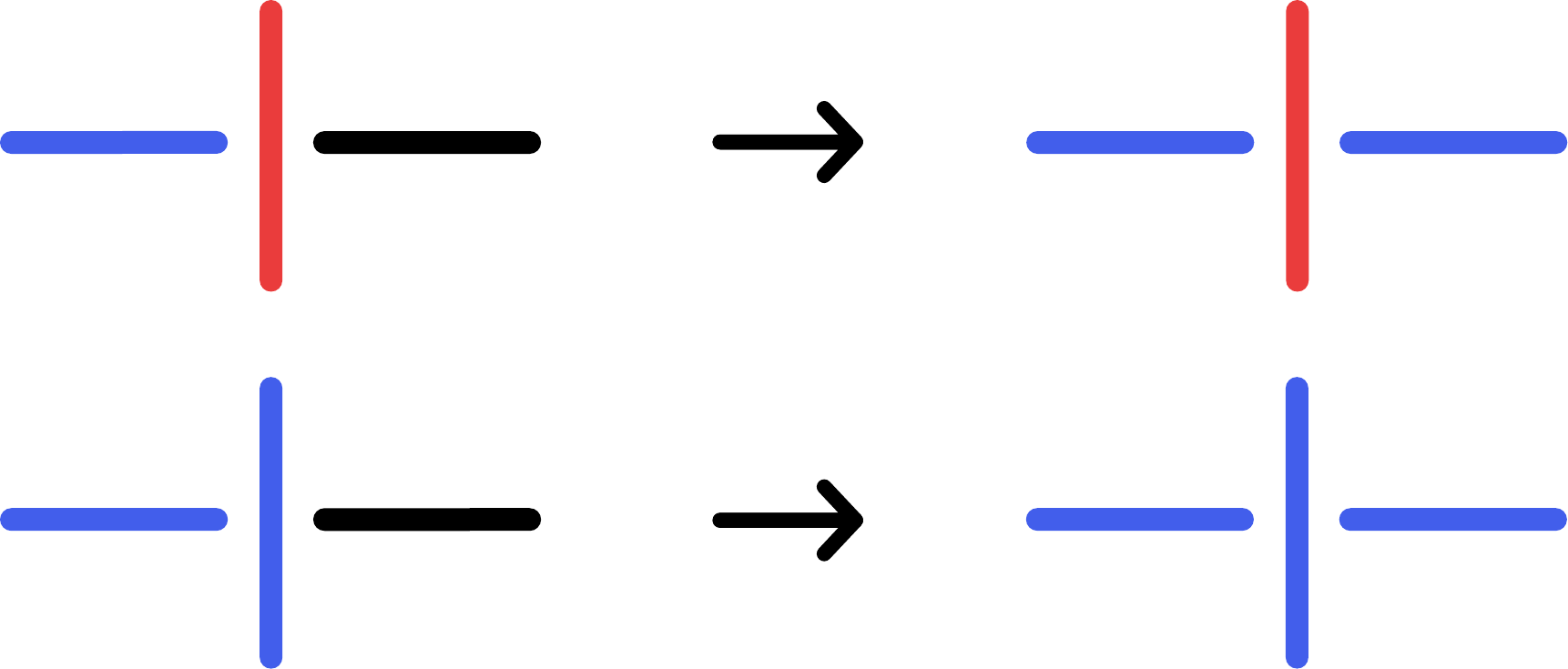}
\centering
\caption{The coloring move: if the overstrand is colored and one of the understrands at the crossing is colored (with possibly the same color), then the move extends the coloring to the other understrand. \label{fig:moves1}}
\end{figure}

We say that a spatial graph diagram is \textit{completely colored} if all free ends of every strand in the diagram have received a color (in particular, once all free ends of a planar tree in a diagram have been colored, we consider the entire tree to be colored).
Given a spatial graph diagram, suppose it can be completely colored by starting with seeds $\{s_1,\ldots, s_n\}$ and extending these colors using the coloring move. 
If $w_i$ is the weight of the seed $s_i$, and  $\tfrac{1}{2}\sum_{i=1}^n w_i =\vcentcolon k$, then we call the diagram \textit{$k$-Wirtinger colorable}. Note that we must generalize the definition in \cite{blair2020wirtinger} to incorporate weights into our definition in order to match the definition we have chosen for bridge index (see \autoref{def:bridge} and the following discussion).

\begin{example}
    See \autoref{fig:sequence1} for an example of a $\tfrac{7}{2}$-Wirtinger colorable diagram for a spatial $\Theta$-graph using three seeds. It is an exercise to check that this diagram cannot be completely colored with only two seeds, and (a slightly harder exercise to see that) furthermore, it cannot be completely colored with three seeds each of weight $2$. Thus $\tfrac{7}{2}$ is actually the smallest value of $k$ for which this diagram is $k$-Wirtinger colorable. 

\begin{figure}[p]
\includegraphics[width=13cm]{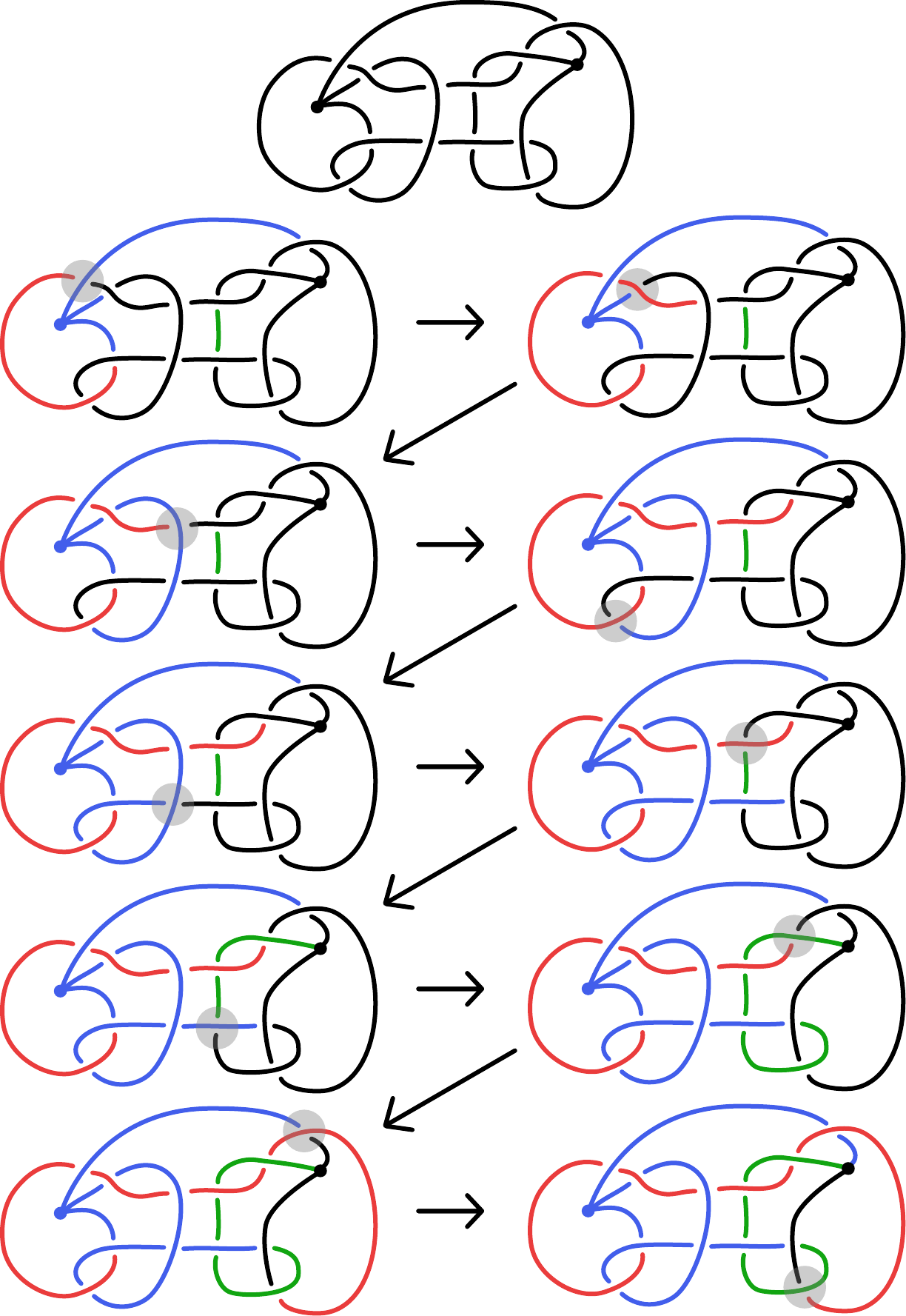}
\centering
\caption{This diagram is $\tfrac{7}{2}$-Wirtinger colorable, using three seeds with weights $2$ (red), $3$ (blue), and $2$ (green). \label{fig:sequence1}}
\end{figure}
\end{example}

\begin{example}
    See \autoref{fig:sequence2} for an example of a $3$-Wirtinger colorable diagram for a spatial $\Theta$-graph using two seeds. It is again an exercise to show that this is the smallest $k$ for which this diagram is $k$-Wirtinger colorable. 
    In fact, our Python code checks for all possible combinations of strands given a fixed number of colors.

\begin{figure}[p]
\includegraphics[width=13cm]{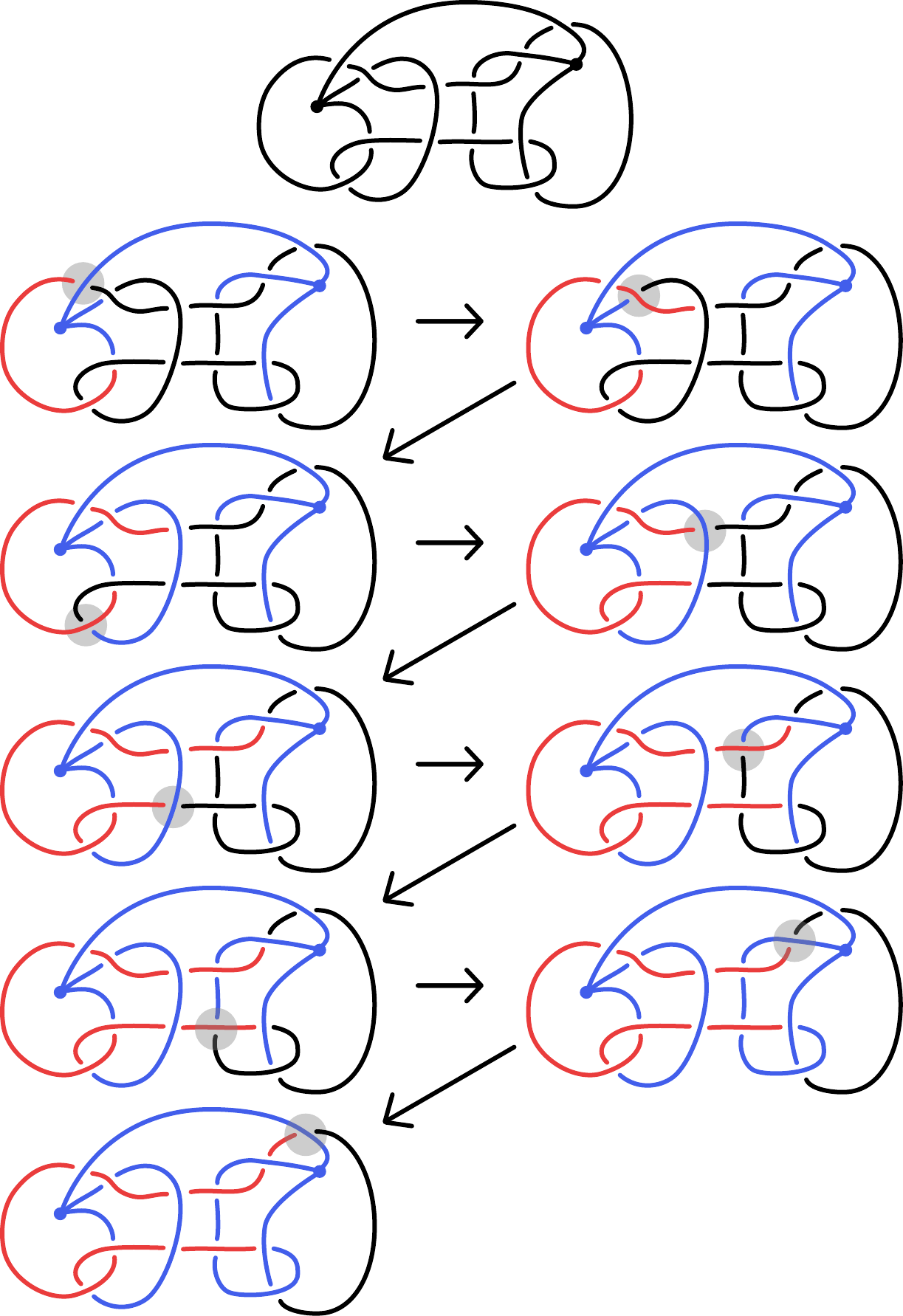}
\centering
\caption{This diagram is $3$-Wirtinger colorable, using two seeds with weights $2$ (red) and $4$ (blue). \label{fig:sequence2}}
\end{figure}
\end{example}

We are now ready to define the Wirtinger number of a spatial graph. 

\begin{definition}    
    The \textit{Wirtinger number of a spatial graph diagram $D$}, denoted $\omega(D)$, is the minimal value of $k$ for which the diagram is $k$-Wirtinger colorable. 
    The \textit{Wirtinger number of a spatial graph $G$}, denoted $\omega(G)$, is the minimal Wirtinger number over all diagrams for $G$. 
\end{definition}

\section{Main result} \label{sec:result}

This section is dedicated to the proof of \autoref{thm:main}. We start with two relevant definitions.

\begin{definition}
    A spatial graph diagram $D$ \textit{realizes} the bridge index for the spatial graph $G$ it represents if a circle can be drawn in the plane in which the diagram sits that splits the diagram into diagrams for two trivial tangles, the tangle diagrams intersect the circle in $b$ points, and the bridge index of $G$ is $\tfrac{1}{2}b$.
    
    Similarly, a spatial graph diagram $D$ \textit{realizes} the Wirtinger number for the spatial graph $G$ it represents if there exist seeds which completely color $D$ and whose weights sum to $2k$, and the Wirtinger number of $G$ is $k$.
\end{definition}

One direction of our main result is more immediate; we consider this direction in the following proposition. We phrase our proofs in this section in terms of spatial graphs in $\mathbb{R}^3$ instead of $S^3=\mathbb{R}^3\cup \{\infty\},$ but this does not make the result less general since the isotopy classes of spatial graphs in $S^3$ and $\mathbb{R}^3$ are the same: the point at infinity $\{\infty\}$ can be taken to be disjoint from a spatial graph and its isotopies. 

\begin{proposition}
     For spatial graphs, the Wirtinger number is at most the bridge index.\label{prop:wirtatmostbr}
\end{proposition}

\begin{proof}
   Let $D$ be a diagram of $G$ that realizes the bridge index. There exists a circle $C$ that intersects $D$ in $\beta(G)$ points that separates $D$ into diagrams of trivial tangles $\tau_1$ and $\tau_2$. 
   We embed $G$ in $\mathbb{R}^3$ such that it projects to $D$ on the $xy$-plane, and adjust the height ($z$-coordinate) of pieces of the graph as follows.   
   After a slight deformation, we can assume that the crossings and vertices lie in different levels. Furthermore, we can arrange the local maxima and vertices to be higher than crossings, and arrange the local minima and vertices to be lower than crossings. We then choose as seeds all of the strands that intersect $C$. The colorings at these seeds extend to the entire diagram because the colorings can be extended one crossing at a time as shown in \autoref{fig:bleed}. Afterwards, all the free ends of every tree will be colored.
\end{proof}

\begin{figure}[ht!]
\includegraphics[width=6cm]{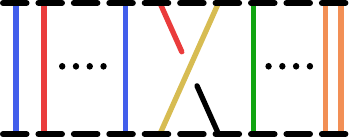}
\centering
\caption{Extending the coloring from all the strands that intersect a projected bridge sphere to all the strands in the trivial tangle $\tau_2$. Extending to $\tau_1$ is similar.
\label{fig:bleed}}
\end{figure}

Now we prove our main result. We define a ``doubling map'' $P \colon \mathscr{G} \rightarrow \mathscr{L}$ from the set $\mathscr{G}$ of diagrams (in $\mathbb{R}^2$) of spatial graphs in $S^3$ to the set $\mathscr{L}$ of diagrams (in $\mathbb{R}^2$) of links in $S^3$, where $P$ doubles each edge, removes a neighborhood of each vertex, and connects up adjacent arcs as shown in \autoref{fig:double}. We will use our doubling map to reduce the case of spatial graphs to the case of links, and then apply the results of \cite{blair2020wirtinger}.

\begin{figure}[ht!]
\includegraphics[width=14cm]{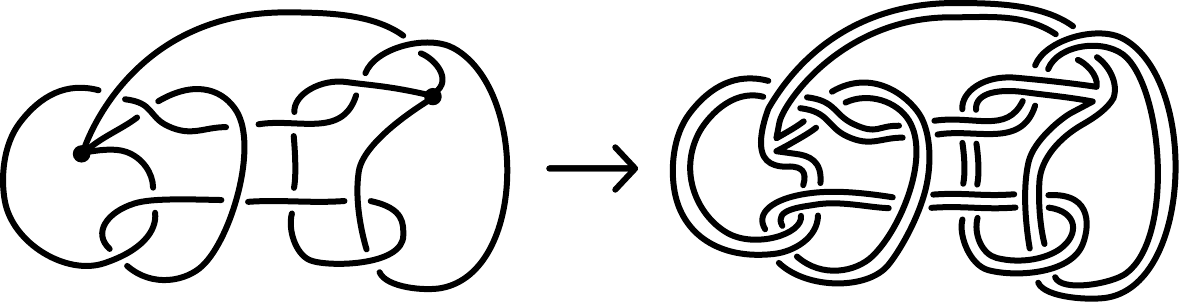}
\centering
\caption{The effect of the doubling map $P$ on a spatial graph diagram.
\label{fig:double}}
\end{figure}

\main*

\begin{proof}
By \autoref{prop:wirtatmostbr} we have that the Wirtinger number is at most the bridge index. 

For the other inequality, suppose we have a spatial graph $G$ with Wirtinger number $k$ which is realized by a diagram $D_G$. Apply $P$ to get $P(D_G) \vcentcolon = D_L$, which is the diagram of a link $L$. Pick seeds $\{s_1,\dots,s_n\}$ in $D_G$ that are sufficient to color the entire diagram in such a way which yields the Wirtinger number $k$. These seeds then induce seeds in $D_L$, but with the rule that different strands in $D_L$ coming from the same seed in $D_G$ must count as different seeds and thus be colored differently, in accordance with the usual procedure for determining seeds for links; see \autoref{fig:double_seeds}. The number of seeds in $D_L$ coming from the $n$ seeds $\{s_1,\dots,s_n\}$ in $D_G$ will be $\sum_{i=1}^m w_i = 2k$, where $w_i$ is the weight of the seed $s_i$ in $D_G$. 

\begin{figure}[ht!]
\includegraphics[width=14cm]{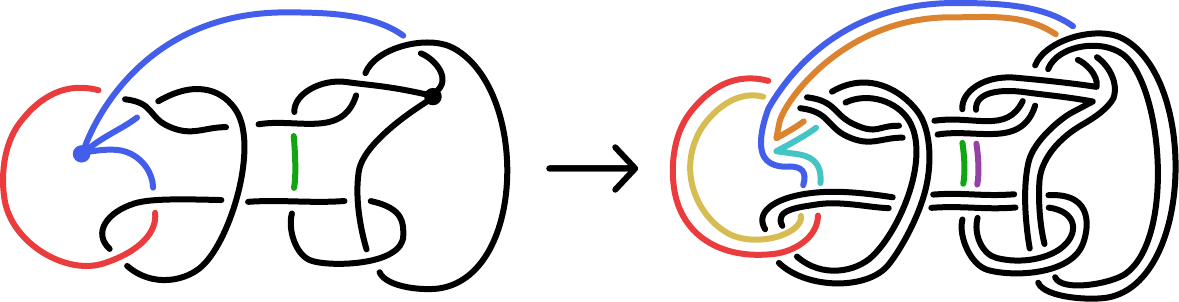}
\centering
\caption{The three seeds in $D_G$ on the left induce seven seeds in $D_L$ on the right.
\label{fig:double_seeds}}
\end{figure}

We then claim that these induced seeds in $D_L$ are sufficient to color the whole diagram. This follows from an analysis of the effect of the map $P$ on the coloring move plus the assumption that the seeds of $D_G$ were sufficient to color all of $D_G$. We define a ``double'' of the original coloring move from \autoref{fig:moves1}, which colors parallel strands one after the next, as shown in \autoref{fig:moves2}. This is a special case of the original coloring move which is made out of repeated applications of the original move in a prescribed way. When coloring our link diagram $D_L$, we will restrict ourselves to only using this double coloring move, so that the coloring will directly correspond to a coloring in the graph diagram $D_G$. Thus we can color all of $D_L$ by finding a sequence of coloring moves in $D_G$ and applying the corresponding doubled coloring moves in $D_L$. Note that once all the free ends of a tree strand in the original diagram are colored, that is sufficient for all strands in the double to be colored; see \autoref{fig:nocross}.

\begin{figure}[ht!]
\includegraphics[width=14cm]{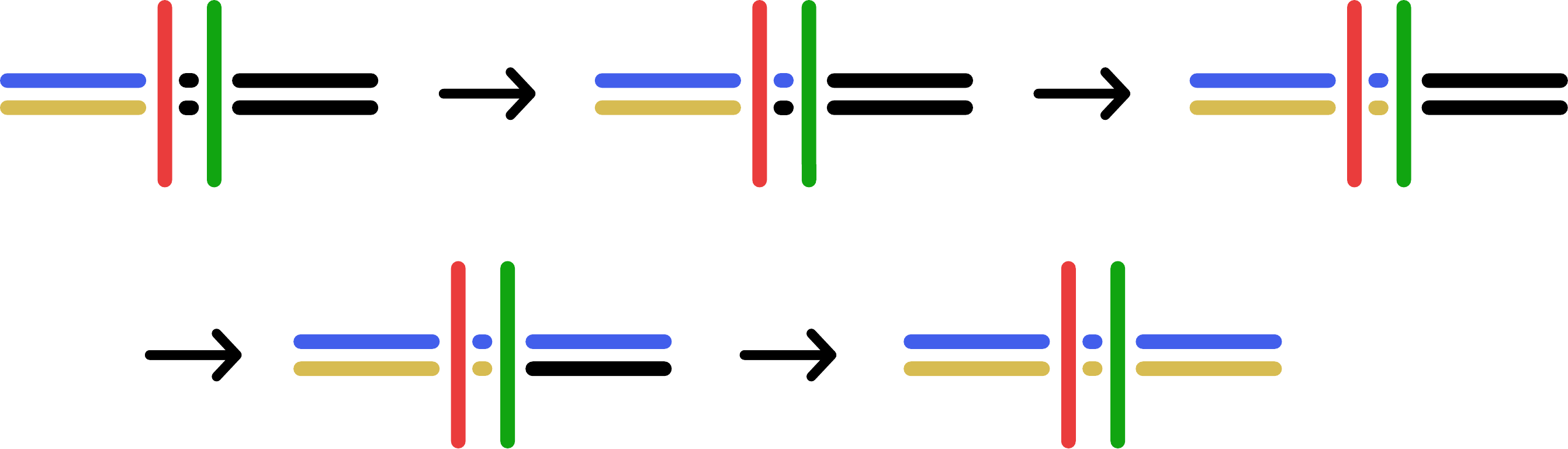}
\centering
\caption{The ``double'' coloring move: if the overstrands are colored (not necessarily the same color) and two of the understrands on the same side of the crossing are colored (not necessarily the same color), then the move extends to the other understrands. 
\label{fig:moves2}}
\end{figure}

\begin{figure}[ht!]
\includegraphics[width=7cm]{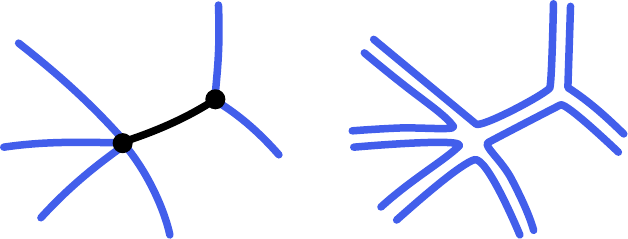}
\centering
\caption{All free ends being colored induces the coloring of every strand in the double. (Note that here we depict all free ends and strands in the double as receiving the same color for the sake of simplicity, but in general this will not be the case.)
\label{fig:nocross}}
\end{figure}

Thus we have a sufficient set of $2k$ seeds for $D_L$, so we may apply the methods of \cite[Theorem~1.3]{blair2020wirtinger} to construct a smooth embedding of $L$ in $\mathbb{R}^3$ with exactly $2k$ local maxima, and conclude that the Wirtinger number of $L$ is bounded below by the bridge index of $L$. 
As before, we embed $G$ in $\mathbb{R}^3$ such that it projects to $D$ on the $xy$-plane, and adjust the height ($z$-coordinate) of pieces of the graph. We embed the seeds at the highest levels, and then embed the remaining strands in the order indicated by the coloring procedure. The result of \cite{blair2020wirtinger} ensures that we do not add new local maxima during the embedding process in addition to the seeds.

However, we also need to embed $L$ in such a way that respects the double structure, so we can ultimately translate back to our original graph. Thus as we embed we require that multiple strands which came from the same seed in $D_G$ must be embedded one after the other (in adjacent sheets), as depicted in \autoref{fig:connectparallel}. Our double coloring move ensures that this will happen as the coloring progresses, so we need only add the requirement that the initial seeds are colored one after the other in this way. 

\begin{figure}[ht!]
\includegraphics[width=12cm]{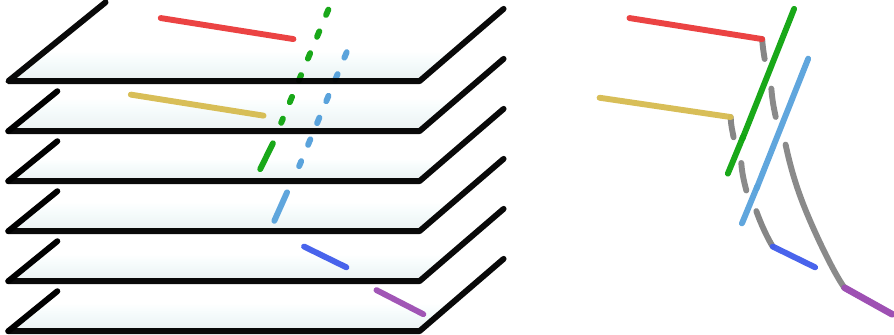}
\centering
\caption{We embed strands which came from the same seed in $D_G$ in adjacent sheets.
\label{fig:connectparallel}}
\end{figure}

From this procedure we get an embedding of each strand of $D_L$ in a plane (parallel to the $xy$-plane) in $\mathbb{R}^3$. Now we connect up the strands in $\mathbb{R}^3$ to form a ribbon link whose core is our spatial graph. We then have an embedding of our spatial graph with bridge index $k$, since when we collapse the ribbon link to its core, any induced seeds in $D_L$ which came from the same seed in $D_G$ will collapse back to one seed. As we started with a Wirtinger number $k$ diagram, we can conclude that the Wirtinger number of $G$ is bounded below by the bridge index of $G$.
\end{proof}

Just as in \cite{blair2020wirtinger}, our main result gives a way to pass between two classical ways to define the bridge index for links: one in terms of the number of overpasses and the other in terms of the number of local maxima. (Our definition for bridge index is in terms of intersections with the bridge sphere, but this is a rephrasing of the maxima definition necessitated by the setting of spatial graphs; see the comments before \autoref{def:bridge}.) We elaborate how to go between these formulations for spatial graphs now.

An \textit{overpass} is a strand in a spatial graph diagram that contains at least one overcrossing (and by definition, contains no undercrossings). We define the \textit{overpass bridge index of a spatial graph diagram} with overpasses $\{O_1,\ldots,O_n\}$ to be $\frac{1}{2}\sum w_i$, where $w_i$ is the weight corresponding to the strand $O_i$. The \textit{overpass bridge index of a spatial graph $G$} is the minimum value of this sum over all diagrams representing $G$.

\begin{corollary}
    The overpass bridge index equals the bridge index.
\end{corollary}

\begin{proof}
    For a given spatial graph, take a diagram realizing the overpass bridge index and declare that the overpasses are seeds. The coloring can surely be extended to the entire diagram since all of the overcrossings are colored.

    For the other inequality, take a bridge splitting realizing the bridge index. By definition, each trivial tangle is a collection of unknotted trees in a $3$-ball. This means that there exists a disk containing each unknotted tree that guides an isotopy of each tree down to the bridge splitting sphere so that at the end of the process, the trees in one trivial tangle $\tau_1$ lie completely in the bridge sphere without intersecting one another. Each component then corresponds to an overpass since performing the same process on the other trivial tangle $\tau_2$ gives another collection of trees on the bridge sphere, and we can declare that any intersection point in the diagram has an arc from $\tau_1$ crossing over an arc from $\tau_2.$
\end{proof}


\section{Lower bounds} \label{sec:bounds}
The following techniques are computable quantities that bound the bridge index of graphs from below. In the setting of links, we know by experimentation that depending on the link types, one method may be more effective than the others \cite{vo2024learning}. That is, in some instances the computer code finds enough quandle colorings, but not a rank bound, and vice versa. 

In general, both techniques require an \textit{oriented} spatial graph as input, meaning that we designate each edge as having a direction coming out of one vertex and into the other, denoted with an arrow. However we can disregard orientation when using involutory quandles (with techniques from \autoref{sec:quandle}) or groups where each generator in the label is its own inverse (with techniques from \autoref{sec:rank}).

\subsection{Quandle colorings} \label{sec:quandle}
In \cite{fleming2007virtual,niebrzydowski2010coloring} the authors generalized knot quandles to spatial graphs.

\begin{definition}\label{def:quandle}
    A \textit{quandle} is a set $X$ equipped with a binary operation $\triangleright$ satisfying the following axioms. 
\end{definition}

\begin{enumerate}
     \item[(1)] $x\triangleright x = x$ for all $x\in X$.
     \item[(2)] The map $f_y:X\rightarrow X$ defined by $f_y(x)=x\triangleright y$ is invertible for all $y \in X$. 
     \item[(3)] $(x\triangleright y) \triangleright z = (x\triangleright z) \triangleright (y\triangleright z)$ for all $x,y,z \in X$.
\end{enumerate}

A \textit{consistent labeling}\footnote{Various authors also use the term \textit{quandle coloring}, \textit{$X$-coloring}, or simply \textit{coloring} for this} of a diagram by a quandle is an assignment of a quandle element to each arc, such that arcs connected to the same vertex receive the same label, and the relation in \autoref{fig:quanrel} is satisfied at each crossing. It can be shown that for knots and links, the number of consistent labelings is unchanged by the Reidemeister moves. For general spatial graphs, however, this is no longer true.

\begin{figure}[ht!]
\labellist
\small\hair 2pt
\pinlabel {$y$} at 73 20
\pinlabel {$x\triangleright y$} at 110 119
\pinlabel {$x$} at -5 20
\endlabellist
\centering
\includegraphics[width=2cm]{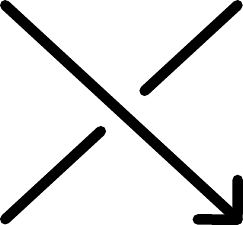}
\caption{The relation at each crossing in a quandle labeling. \label{fig:quanrel}}
\end{figure}

The move for spatial graphs which interacts with the vertices, as seen in \autoref{fig:nottricol}, forces the quandles that are suitable for spatial graphs to have a more restrictive form. For instance, the coloring by the dihedral quandle $R_3$ (Fox 3-coloring) no longer works for some graphs, such as those containing the situation in \autoref{fig:nottricol}, due to this new graph move. 

\begin{figure}[ht!]
\centering
\includegraphics[width=6cm]{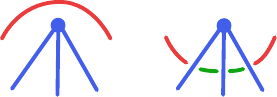}
\caption{For a quandle coloring to work, the strands near a crossing need to all receive the same color or all distinct colors. On the left, the strands are disjoint and the arcs connected to the same vertex receive one color. If the other strand receives a different color, after we perform a move that does not change the graph type, we are forced to have a crossing that uses only two colors.}\label{fig:nottricol}
\end{figure}

On the other hand, the order four Alexander quandle with the following multiplication table works for the new graph move on a degree three vertex. 
\begin{align*} 
0 = 0 \triangleright 0 = 1 \triangleright 2 = 2 \triangleright 3 = 3 \triangleright 1 \\
1 = 0 \triangleright 3 = 1 \triangleright 1 = 2 \triangleright 0 = 3 \triangleright 2 \\
2 = 0 \triangleright 1 = 1 \triangleright 3 = 2 \triangleright 2 = 3 \triangleright 0 \\
3 = 0 \triangleright 2 = 1 \triangleright 0 = 2 \triangleright 1 = 3 \triangleright 3
\end{align*}

This motivates the following definitions.

 \begin{definition}\label{def:nquandle}
    An \textit{$n$-quandle} is a set $X$ equipped with a binary operation $\triangleright$ satisfying the quandle axioms in \autoref{def:quandle} and one additional axiom $((x\triangleright y)\triangleright y) \triangleright \cdots \triangleright y)=x$ for all $x,y\in X$, where the total number of quandle operations in the additional axiom is $n$. 
\end{definition}

 \begin{definition}\label{def:coloring}
    Give a finite $n$-quandle $X$ and a spatial graph $G$, the \textit{coloring number} $Col_X(G)$ is the number of consistent labelings of $G$ by $X$.
\end{definition}

 Recall that when we compute the bridge index, which is shown to be equal to the Wirtinger number by \autoref{thm:main}, we assign each seed a weight. The traditional quandle counting invariant will give a bound on the number of seeds unweighted. When we apply it for practical purposes in \autoref{sec:examples}, we will add in the weight consideration. 

\begin{definition} \label{def:unweighted}
    Given a spatial graph $G$, define the \textit{unweighted bridge index} $\widehat{\beta}(G)$ to be the minimum number of seeds over all diagrams of $G$ (with no weights assigned).
\end{definition}

\begin{proposition} \label{prop:color}
    Let $X$ be a finite $n$-quandle of order $|X|$. Then $Col_X(G)\leq |X|^{\widehat{\beta}(G)}$. In other words, $$\log_{|X|} Col_X(G)\leq \widehat{\beta}(G).$$
\end{proposition}

\begin{proof}
By \autoref{thm:main}, there exists a diagram $D$ realizing the $\widehat{\beta}(G)$ seeds. There are $|X|$ choices of quandle elements from $X$ to assign to each seed. The coloring move in \autoref{fig:moves1} can be regarded as a less restricted version of the quandle rule in \autoref{fig:quanrel}. More precisely, note that the rule in \autoref{fig:quanrel} says that the outgoing quandle labeling is a result of using the quandle operation on the incoming arc and the overstrand, which is analogous to saying that the outgoing arc receives a coloring once the incoming arc and the overstrand receive colors in \autoref{fig:moves1}. Thus, there are $|X|^{\widehat{\beta}(G)}$ possibilities of quandle labels for the seeds that will generate the quandle labeling for the entire diagram. The reason why $Col_X(G)\leq |X|^{\widehat{\beta}(G)}$ is possibly not an equality is because not every labeling generated from the seeds is consistent.
\end{proof}

We now provide a lower bound for the bridge index of vertex sums of spatial graphs (recall \autoref{sec:graphs}) in terms of quandles, which will be useful to us in \autoref{sec:examples}. The argument is based on \cite{clark2016quandle}. The key is the existence of quandles called \textit{homogeneous quandles}, which have the property that for any quandle elements $x,y,$ there is an automorphism $h$ such that $h(x)=y.$ Note that the map $f_y$ sending $x$ to $x\triangleright y$ in the second axiom of \autoref{def:quandle} is an example of an automorphism by the way the axiom is stated. The reader can check from the table before \autoref{def:nquandle} that the Alexander quandle of order four is homogeneous.

\begin{proposition} \label{prop:homogeneous}
Let $X$ be a homogeneous $n$-quandle, and $G\#_n G'$ be the $n$-valent vertex sum of spatial graphs $G$ and $G'$. Then $$Col_X(G\#_n G')= \frac{1}{|X|}Col_X(G)\cdot  Col_X(G').$$ Thus, we achieve the lower bound $$\log_{|X|} Col_X(G) + \log_{|X|} Col_X(G') - 1 \leq \widehat{\beta}(G\#_n G').$$
\end{proposition}
 
\begin{proof}
    The number of colorings where a fixed color $x$ is assigned to the arcs adjacent to a vertex $v$ is precisely $Col_X(G)/|X|$. To see this, observe that the number of colorings $Col_X(G)=\sum_y Col_X(G,y)$, where $Col_X(G,y)$ denotes the number of colorings such that the arcs near $v$ get the label $y$. We can get another coloring in which the arcs adjacent to $v$ are colored $z\neq y$ by applying an automorphism $h$ such that $h(y)=z,$ which exists by the homogeneity assumption. Thus, $Col_X(G)=|X|Col_X(G,y)$.

    Now we consider the vertices of $G$ and $G'$ that we will do the vertex sum along. For any fixed coloring of $G$, we look at the coloring at the vertex; call it $x$. Then there are $\frac{1}{|X|}Col_X(G')$ colorings for $G'$ that have the coloring $x$ at the vertex to match up with $G$. Therefore, $Col_X(G\#_n G')= \frac{1}{|X|}Col_X(G)\cdot  Col_X(G')$. The stated inequality comes from combining this equation with the result of \autoref{prop:color} and simplifying.
\end{proof}

The condition that the colors of the arcs adjacent to a vertex have to be the same can be relaxed. In fact, this may be needed to bound our definition of the bridge index in which we count the number of intersections of the spatial graph with bridge spheres (see \autoref{def:bridge} and the following discussion). This is pointed out in \cite{ishii1997color} with a brief argument and a reference to the proof in a different paper.\footnote{To the best of the authors' knowledge, this other paper never appeared in the literature.} We provide the proof again here for accessibility and completeness, as some examples in \autoref{sec:examples} will need this type of coloring.

\begin{definition} \label{def:gen}
A \textit{generalized $3$-coloring of a $4$-regular spatial graph diagram} is an assignment of three elements $\{0,1,2\}$ to the arcs such that the following two properties hold.
\begin{enumerate}
    \item If $x,y,$ and $z$ are the incoming arc, outgoing arc, and the overstrand at a crossing respectively, then $2z-y = x \pmod 3$.
    \item Let $a_1,a_2,a_3,a_4$ be labels around a vertex arranged in a counterclockwise fashion. Then $a_1-a_2+a_3-a_4 = 0 \pmod 3$.
\end{enumerate}
\end{definition}

\begin{proposition} \label{prop:gen}
    The number of generalized $3$-colorings of a $4$-regular spatial graph diagram is a spatial graph invariant.
\end{proposition}

\begin{proof}
    The condition around each crossing is satisfied because tricolorability is known to be a knot invariant. At a vertex, we have to check that there is a bijection before and after Reidemeister moves. In \autoref{fig:quandmorecol}, all the Reidemeister moves involving vertices are depicted. The labels after a chain of quandle rules for the top left picture will give $y=2a-x, z=2b-y, w=2c-z,$ and $t=2d-w.$ In other words, $t=2d-(2c-(2b-(2a-x))) = -2(a-b+c-d)+x = x$ since $a-b+c-d = 0 \pmod 3.$ For the other version of the same move, we have that $e=2x-a, f=2x-b,g=2x-c,$ and $h=2x-d$. Since we have $e-f+g-h=(2x-a)-(2x-b)+(2x-c)-(2x-d) = 0 \pmod 3$, the coloring is invariant.

    For the twisting move at a vertex, we check that the sum $j-i+l-r = 0,$ but $r= 2l-k$ so that $j-i+l-2l+k = j-i-l+k = 0.$ 
\end{proof}

\begin{figure}[ht!]
\labellist
\footnotesize\hair 2pt
\pinlabel  {$a$} at 8 15
\pinlabel  {$b$} at 43 0
\pinlabel  {$c$} at 98 0
\pinlabel  {$d$} at 135 15
\pinlabel  {$x$} at 2 46
\pinlabel  {$y$} at 43 25
\pinlabel  {$z$} at 70 22
\pinlabel  {$w$} at 97 25
\pinlabel  {$t$} at 134 46

\pinlabel  {$a$} at 180 15
\pinlabel  {$b$} at 216 0
\pinlabel  {$c$} at 270 0
\pinlabel  {$d$} at 307 15
\pinlabel  {$x$} at 173 50
\pinlabel  {$x$} at 308 50

\pinlabel  {$a$} at 351 15
\pinlabel  {$b$} at 388 0
\pinlabel  {$c$} at 442 0
\pinlabel  {$d$} at 479 15
\pinlabel  {$x$} at 347 46
\pinlabel  {$e$} at 387 60
\pinlabel  {$f$} at 398 46
\pinlabel  {$g$} at 430 44
\pinlabel  {$h$} at 439 63

\endlabellist
\centering
\includegraphics[width=11cm]{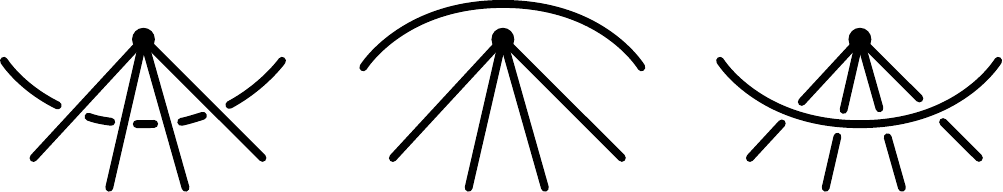}
\vspace{2em}

\labellist
\small\hair 2pt
\pinlabel  {$i$} at 20 1
\pinlabel  {$j$} at 5 94
\pinlabel  {$k$} at 110 1
\pinlabel  {$\ell$} at 120 94

\pinlabel  {$i$} at 197 1
\pinlabel  {$j$} at 180 94
\pinlabel  {$k$} at 332 1
\pinlabel  {$\ell$} at 347 94
\pinlabel  {$r$} at 300 94
\endlabellist
\includegraphics[width=6cm]{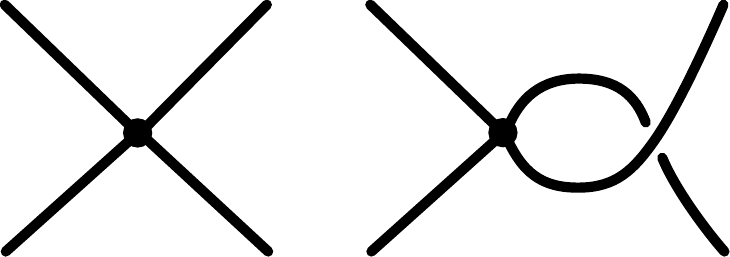}
\caption{The number of generalized tricolorings is preserved by Reidemeister moves involving vertices. \label{fig:quandmorecol}}
\end{figure}

\begin{remark}
McAtee, Silver, and Williams consider colorings of spatial graphs using topological groups in \cite{mcatee2001coloring}. Since any topological group is a quandle \cite{rubinsztein2007topological}, it would be interesting to combine the quandle techniques discussed in this section with the techniques of \cite{mcatee2001coloring}. 
\end{remark}

\subsection{Meridional rank} \label{sec:rank}
It is a straightforward exercise to see that any group produces a quandle structure given by $a\triangleright b = b^{-1}ab$. In fact, one motivation for quandle theory is to generalize the notion of conjugation in group theory. In the previous subsection, we counted the number of quandle colorings to get a bound on the number of seeds. When the quandle comes from a group, and the group is well-understood in terms of how to generate it, it suffices to find just one coloring (we do not have to count all possible colorings) to get a bound on the number of seeds. This is what we do in this subsection.

A presentation for the fundamental group of the complement of a spatial graph $G$ in $S^3$ can be computed via the Wirtinger algorithm in a similar way to the computations for knots and links (this can be proven using Van Kampen's theorem). Each arc of a spatial graph diagram gives a generator, and each crossing gives a relation of the form 
$x_ix_jx^{-1}_ix^{-1}_k$.
At a vertex, we have a new type of relation of the form $x_{i_1}^{\epsilon_1}x_{i_2}^{\epsilon_2}\cdots x_{i_n}^{\epsilon_n}$, where the generators are listed in order as we go clockwise around a vertex, and $\epsilon_i$ is 1 (resp. $-1$) if the arc is directed into the vertex (resp. directed out of the vertex). See \autoref{fig:wirtgraph}.

\begin{figure}[ht!]
\labellist
\small\hair 2pt
\pinlabel {$x_{i}$} at 78 20
\pinlabel {$x_{j}$} at 78 63
\pinlabel {$x_{k}$} at -5 20
\pinlabel {$x_{i_1}$} [tr] at 155 36
\pinlabel {$x_{i_2}$} [tr] at 190 36
\pinlabel {$x_{i_n}$} [tr] at 265 36
\endlabellist
\centering
\includegraphics[width=6cm]{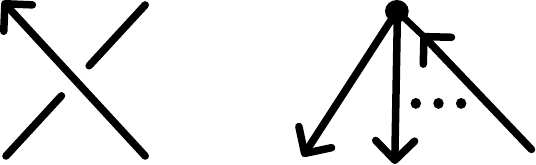}
\caption{On the left is the relation 
$x_ix_jx^{-1}_ix^{-1}_k$ at each crossing. On the right is the relation $x_{i_1}^{\epsilon_1}x_{i_2}^{\epsilon_2}\cdots x_{i_n}^{\epsilon_n}$ at a vertex.\label{fig:wirtgraph}}
\end{figure}

The technique of bounding the number of handles of knotted objects by the meridional rank of the fundamental group has been used before in other settings such as classical links \cite{baader2021coxeter} and surface-links \cite{joseph2024meridional}. We now adapt the method to the setting of spatial graphs.

In \cite{livingston1995knotted} Livingston defined the \textit{vertex constant group} $\pi^*$ to be the quotient of $\pi_1(S^3\backslash G)$ defined by setting all meridians at each vertex equal to each other. We are interested in bounding the rank of the vertex constant group, where our generating set consists of meridians. We achieve this by searching for a consistent labeling of our diagram by reflections in a Coxeter group $C(\Delta)$. This will produce a surjective homomorphism from $\pi^*$ to $C(\Delta)$.

\begin{definition} \label{def:Coxeter}
    Recall that a presentation of a \textit{Coxeter group} $C(\Delta)$ can be obtained from a weighted graph $\Delta$ as follows. Label the vertices of $\Delta$ as $x_1,x_2,\ldots, x_n$, which will be the generators. If there is an edge with weight $w$ connecting $x_i$ to $x_j$, then we have a relation $(x_ix_j)^w$ in the presentation. Additionally, the relations also include $x_i^2,$ for $i=1,2,\ldots, n.$
\end{definition}

The number of vertices of $\Delta$ is referred to as the \textit{reflection rank}. The following lemma is important in obtaining our bound. 

\begin{lemma}[{\cite[Lem.~2.1]{felikson2010reflection}}]
    If all edge weights of $\Delta$ are at least $2$, then the reflection rank equals the minimal number of reflections needed to generate $C(\Delta).$\label{lem:algebraresult}
\end{lemma}

\begin{definition}
    A labeling of arcs in a spatial graph diagram $D$ by a Coxeter group is \textit{consistent} if at each crossing where the overstrand is labeled with an element $g$ and the two understrands are labeled $h$ and $k$, we have that $ghg = k$. We say that the label \textit{generates} the group if every element of $C(\Delta)$ can be written as a product of the elements that appear as labels along with their inverses.
\end{definition}

Note that in Coxeter groups the generators are involutions, so the role of $h$ and $k$ can be interchanged in the definition above.

\begin{proposition}
If a diagram $D$ representing a spatial graph $G$ can be labeled by $n$ reflections from a Coxeter group $C(\Delta)$ with reflection rank $n$, then $\widehat{\beta}(G)\geq n.$\label{prop:reflection}
\end{proposition}

\begin{proof}
        In the same spirit as \autoref{prop:homogeneous}, the labels on seeds determine all the other labels. The difference is that we use a labeling by reflections in Coxeter groups in this setting. Thus, these labels generate $C(\Delta)$. We then appeal to \autoref{lem:algebraresult} to conclude that we need at least $n$ reflections to generate $C(\Delta)$.
\end{proof}

\section{Computations} \label{sec:examples}

In this section, we use our techniques to estimate or provide exact values for the bridge indices of many different examples of knotted spatial graphs, many of which are almost unknotted. Additionally, in \autoref{sec:code} we discuss the particulars of our code and in \autoref{sec:large} we prove \autoref{thm:index}. 
Refer to \autoref{sec:graphs} for a review of the spatial graph terminology we use throughout this section. We will use the term ``coloring'' for both Wirtinger coloring and quandle coloring, but the type of coloring should be clear from context. When calculating lower bounds, we will sometimes refer to both the (weighted) bridge index and the unweighted bridge index (recall \autoref{def:unweighted}).

We break into various subsections according to the type of construction, but start first with a relatively simple example.

\begin{example}[Almost unknotted graphs from unlinks] \label{ex:Kinoshita}
As a warmup, \autoref{fig:ex_arc} gives an almost unknotted handcuff graph (left) and an almost unknotted $\Theta$-graph (right). Both of these well-known examples -- for instance the $\Theta$-graph is the Kinoshita graph \cite{kinoshita1958alexander,kinoshita1972elementary} --  were constructed by adding an edge to an unlink, and both have bridge index $\tfrac{5}{2}$. These diagrams are already in bridge position demonstrating this index, but this can also be verified by finding $\tfrac{5}{2}$-Wirtinger colorings of the diagrams. 

Since there is a nontrivial coloring by the Alexander quandle of order four from \autoref{sec:quandle}, the unweighted bridge index, and thus the bridge index, is bounded below by $2$.
By analyzing the various possibilities for how the vertices are placed in the bridge splitting, \`{a} la the arguments found in \autoref{ex:Suzuki} or \autoref{ex:vertex2}, one can -- without too much difficulty -- raise the lower bound of the bridge index to $\tfrac{5}{2}$, as desired. 

Note that the number of crossings in each column of twists can be systematically increased and we still retain the consistency of the coloring. By systematic, we mean that if there are three crossings originally, we can keep adding a multiple of three crossings.

\begin{figure}[ht!]
\labellist
\small\hair 2pt
\pinlabel $0$ at 69 220
\pinlabel $0$ at 380 220
\pinlabel $0$ at 69 75
\pinlabel $2$ at 190 135
\pinlabel $2$ at 300 135
\pinlabel $3$ at 493 145
\pinlabel $2$ at 493 95
\pinlabel $2$ at -12 135
\pinlabel $1$ at 10 10
\pinlabel $1$ at 320 10
\pinlabel $1$ at 5 280
\pinlabel $1$ at 315 280
\endlabellist
\includegraphics[width=7cm]{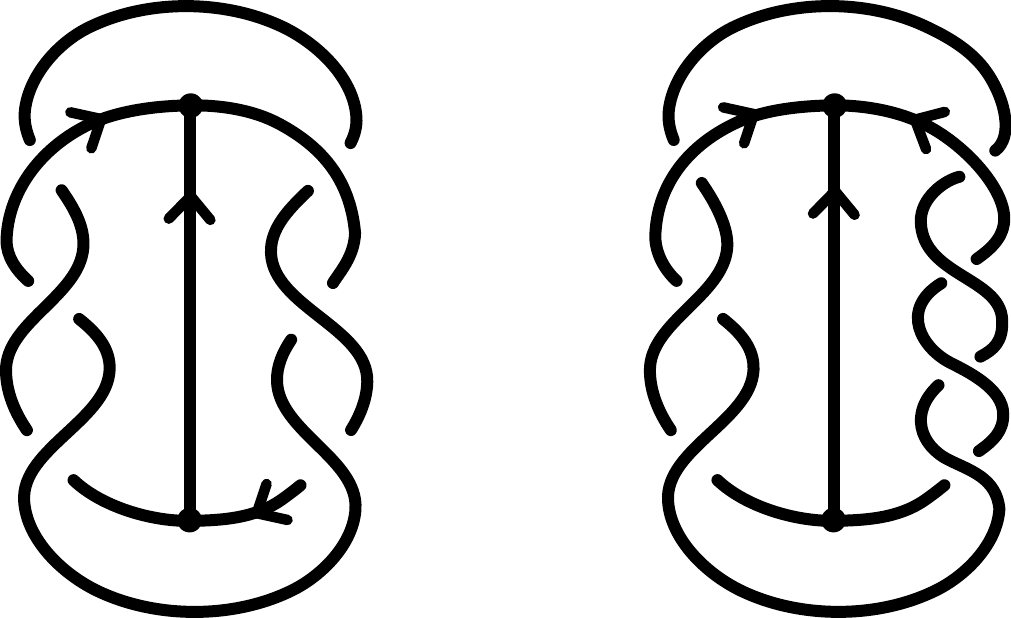}
\centering
\caption{A quandle colored almost unknotted handcuff graph (left) and a quandle colored almost unknotted $\Theta$-graph (right), both with bridge index $\tfrac{5}{2}$.\label{fig:ex_arc}}
\end{figure}
\end{example}

\subsection{Python code} \label{sec:code}

We have implemented our algorithm in Python to compute upper bounds for the Wirtinger numbers of arbitrary (knotted) $\Theta_n$-graphs ($n\le 26)$.  
Our algorithm generalizes coloring algorithms used to compute bounds for the Wirtinger numbers for knots and links. We refer readers to \cite{blair2020wirtinger, lee2024widths} for details on the knot and link cases. Our Python code can be found at 
\href{https://github.com/hanhv/graph-wirt}{https://github.com/hanhv/graph-wirt}.
While the full detail can be found in the URL, we also give a summary here. 

\subsubsection*{Notation setup}
Due to the lack of spatial graph tabulations, we generate Gauss codes of some spatial graphs from those of links. 
\begin{definition}
    A \textit{Gauss code for an $m$-component link} is a list of $m$ lists of positive and negative integers such that for all $k\ne0, |k|\le m$, both $k$ and $-k$ appear once each among the lists.
\end{definition}
\begin{definition}
A \textit{Gauss code for a spatial graph} is a list of \( n \) lists, where each list corresponds to an edge of the graph. Each edge joins two vertices and includes information about the crossings it passes through.
\end{definition}
For \(n \leq 26\), a Gauss code for a \(\Theta_n\)-graph contains \( n \) lists, each corresponding to an edge in the graph. Each edge of the graph is represented by a list of numbers and symbols in the following way. 
\begin{itemize}
    \item The list starts and ends with a symbol in the form of a letter-number pair, such as \texttt{a1} or \texttt{b2}. The letter designates the direction, represented by alphabet letters, and the number following the letter indicates a vertex of the graph. The vertex number is positive.
    \item The numbers \( |k| \le m \), where \( k \) is non-zero and different from the vertex number, correspond to crossings.
    For a valid Gauss code, every integer from $k\le |m|$, excluding zero and the vertex numbers, must appear exactly once across the entire code.  
\end{itemize}
For example, the following Gauss code 
\begin{verbatim}
gauss_code = [
    [a1, 3, -4, a2],
    [b1, -5, 6, 7, -8, 9, -7, 10, -11, 12, -10, -13, 14, b2],
    [c1, -9, 8, -3, 5, -6, 13, -14, 4, -12, 11, c2]
]
\end{verbatim}
corresponds to the graph shown in \autoref{fig:example_gc}. 

\begin{figure}[ht!]
\labellist
\small\hair 2pt
\pinlabel $1$ at 28 65
\pinlabel $a$ at 37 85
\pinlabel $b$ at 66 75
\pinlabel $c$ at 52 58
\pinlabel $2$ at 218 90
\pinlabel $a$ at 220 105
\pinlabel $b$ at 195 106
\pinlabel $c$ at 200 80

\pinlabel $3$ at 53 107
\pinlabel $4$ at 201 130
\pinlabel $5$ at 71 99
\pinlabel $6$ at 122 90
\pinlabel $7$ at 124 53
\pinlabel $8$ at 55 7
\pinlabel $9$ at 81 53
\pinlabel $10$ at 152 53
\pinlabel $11$ at 182 53
\pinlabel $12$ at 172 4
\pinlabel $13$ at 153 71
\pinlabel $14$ at 165 111
\endlabellist
\includegraphics[width=7cm]{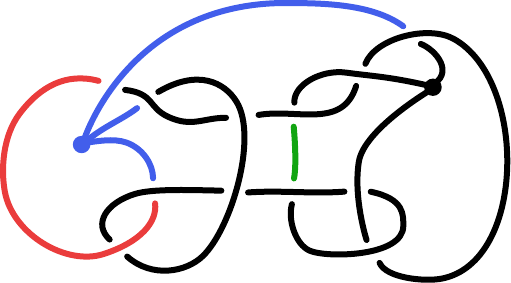}
\centering
\caption{Example of a Gauss code. This diagram is $\tfrac{7}{2}$-Wirtinger colorable using the seeds \texttt{[[a1, b1, c1], [-10, -13], [-9, 8, -3]]}, which are colored blue, green, and red respectively. \label{fig:example_gc}}
\end{figure}

\subsubsection*{Generating Gauss codes of $\Theta_4$-graphs from two-component links}
First we extract two-component link diagrams from the link table provided by SnapPy \cite{SnapPy}. Each of these links takes the form of a Gauss code made up of two lists of numbers. Selecting a two-component link, by ``singularizing'' we turn two crossings between components into 4-valent vertices. 
Here, a pair of crossings \((x, y)\) is considered \textit{singularizable} if both \(x\) and \(y\) appear in one component, while \(-x\) and \(-y\) appear in the other component. 
We only choose singularizable pairs that are of distance at least four apart (this is optional), ensuring that each edge of the graph contains at least four crossings. 
By singularizing a pair of crossings \(i\) and \(j\), we split each component of the link, viewed as a cyclic list, into two parts with endpoints at \(i\) and \(j\). 
This has the effect of turning the two-component link into a $\Theta_4$-graph. 

\subsubsection*{Computing the Wirtinger number of a graph diagram}

In the context of our code, it is convenient to \textit{truncate} our spatial graph diagrams, meaning we imagine removing small neighborhoods of vertices (of degree larger than two). We refer to these neighborhoods as \textit{pods}. After removing pods, we have two types of arcs: arcs which connect the endpoints of pods to undercrossings, and arcs which connect two undercrossings.

The program begins by identifying all strands in the Gauss code which can be classified as pods or one of the two aforementioned types of arcs. It first sets \(k=1\) and generates a list of \textit{seeds},\footnote{Note that elsewhere in the paper we refer to \textit{each} initially colored strand as a seed, but in the context of our code, we refer to the \textit{entire} collection of initial strands as a seed.} which are combinations of \(k\) strands, ensuring that each strand must be either a pod or an arc which connects two undercrossings
(more precisely, arcs connecting the endpoints of pods to undercrossings are not allowed to be in a seed).
For each seed, it colors the strands in the seed and creates a truncated version of the graph by removing any pods that are not part of the selected seed. 

The program then attempts to maximally extend this truncated graph through two types of coloring moves: one that extends from the endpoint of a pod and another that extends from an undercrossing if its corresponding overstrand is colored. If the entire truncated graph can be colored, the seed is added to the list of \textit{colorable seeds}. Half of the sum of weights of strands in the seeds gives an upper bound for the Wirtinger number and an upper bound for $k$ (the maximum number of strands in a colorable seed).
The program keeps running until $k$ reaches its upper bound and returns the minimum Wirtinger number and the corresponding seeds.

\subsubsection*{Output}
We extracted and computed the Wirtinger numbers for over 10,000 $\Theta_4$-graph diagrams.

\subsection{Hara's graphs and modifications}
\begin{example}[Hara's graphs]
The family of almost unknotted $\Theta_{4}$-graphs given by Hara \cite[Fig.~2]{hara1991symmetry} all have unweighted bridge index $2$, and (weighted) bridge index at most $3$, which can be quickly verified by finding Wirtinger colorings of the diagrams (using the same seeds for both versions of bridge index). This is perhaps surprising, as these particular diagrams have many more local maxima and minima.
\end{example}

\begin{example}[Modifications of Hara's graphs]
We create two families of spatial graphs by modifying Hara's $\Theta_{4}$-graphs from the previous example; however these families are \textit{not} almost unknotted. For both families we give upper bounds for their bridge indices, which can be verified by finding corresponding Wirtinger colorings. In both cases, we conjecture that this upper bound is in fact the bridge index. 

For the sake of simplicity, we depict both families as coming from $\Theta$-graphs rather than $\Theta_{4}$-graphs, but both could be further extended by adding more edges in analogous ways. This would likely not change the bridge index. Note that with this change, we can (perhaps more appropriately) think of these families as modifications of a Kinoshita-Wolcott graph \cite{kinoshita1958alexander,kinoshita1972elementary,wolcott1987knotting} (see \cite[Fig.~1]{jang2016new}).

First we alter Hara's graphs by replacing the braids on two strands with braids on $n$ strands as in \autoref{fig:ex_br_par} (left). In the specific instance shown in \autoref{fig:ex_br_par} (left) the constituent knots are all $5_2$.

\begin{fact}
    Let $G_n$ be the $\Theta$-graph obtained from Hara's graphs by replacing the braids on two strands with braids on $n$ strands. Then the unweighted bridge index $\widehat{\beta}(G_n)$ is less than or equal to $n$, and the bridge index $\beta(G_n)$ is less than or equal to $\tfrac{2n+1}{2}$.
\end{fact}

Second we alter Hara's graphs by taking $n$ parallel copies of each edge as in \autoref{fig:ex_br_par} (right). For the instance shown in \autoref{fig:ex_br_par} (right) every pair of edges forms an unknot, but not every proper subgraph is trivial: when one edge is deleted, the remaining parallel edge still causes a problem. 

\begin{fact}
    Let $G_n$ be the $\Theta_{3n}$-graph obtained from Hara's graphs by taking $n$ parallel copies of each edge.
    Then the unweighted bridge index $\widehat{\beta}(G_n)$ is less than or equal to $n+1$, and the bridge index $\beta(G_n)$ is less than or equal to $\tfrac{5n}{2}$.
\end{fact}

\begin{figure}[t!]
\includegraphics[width=10cm]{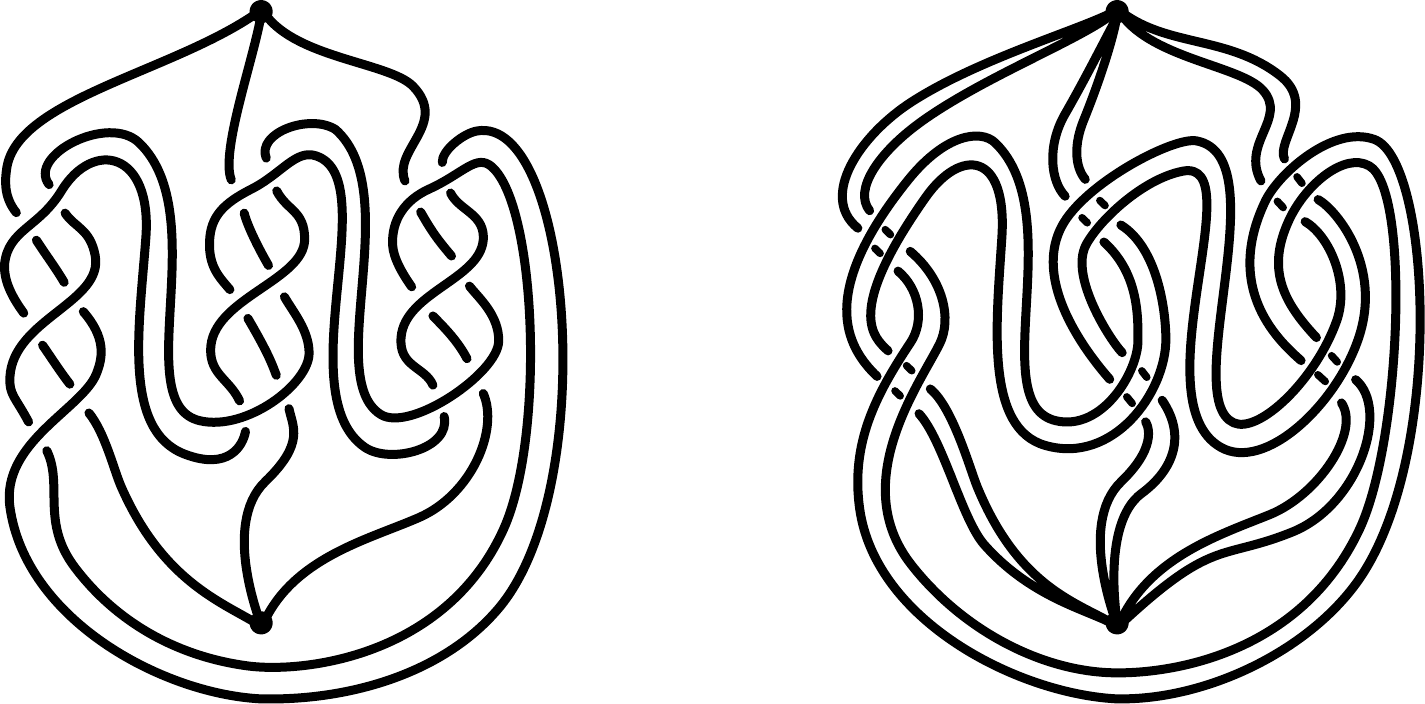}
\centering
\caption{Some spatial graphs constructed by modifying Hara's graphs. On the left we modify by changing the number of strands in the twist regions, and on the right we modify by adding parallel edges. \label{fig:ex_br_par}}
\end{figure}
\end{example}

\subsection{Families from vertex sums} 

\begin{example}[Vertex sums of Suzuki's graph] \label{ex:Suzuki}
Take $n$ copies of the almost unknotted Suzuki $\Theta_4$-graph \cite{suzuki1984almost}, as shown in \autoref{fig:vertexred}. Consider the quandle $(\{0,1,2\},\triangleright)$, where $x\triangleright y=2y-x \pmod 3$. It can be shown that this quandle is homogeneous (see \cite{furuki2024homogeneous}, for example). 
In \autoref{fig:vertexred} we give three consistent colorings of the graph by this quandle; then by permuting the colors, we have a total of nine possible colorings. 

Taking the vertex sum of $n$ copies of the graph, we obtain a diagram with bridge index $n+1$. Using \autoref{prop:homogeneous}, we get that the number of colorings of the resulting vertex sum is $3^{n+1}$. Then by \autoref{prop:color}, the minimum number of seeds is precisely $n+1$, where in this case each trivial tangle contains a vertex. Thus we have shown the following.

\begin{corollary} \label{cor:Suzuki1}
   Let $G_n$ be the $\Theta_4$-graph obtained by taking the vertex sum of $n$ copies of Suzuki's graph. Then the unweighted bridge index $\widehat{\beta}(G_n)$ is $n+1$.
\end{corollary}

With generalized 3-coloring (recall \autoref{def:gen} and \autoref{prop:gen}), we can give a lower bound for the (weighted) bridge index $\beta$, not just the unweighted version $\widehat{\beta}$. The other type of bridge splitting that we have to consider is the case where two vertices of the $\Theta_4$-curve are on the same side. That is, one trivial tangle is a collection of intervals with no vertices.

\begin{figure}[ht!]
\labellist
\small\hair 2pt

\pinlabel $x$ at 90 189
\pinlabel $y$ at -2 150
\pinlabel $z$ at 62 125
\pinlabel $t$ at 83 125
\pinlabel $w$ at 90 150

\pinlabel $f$ at 30 105
\pinlabel $a$ at 13 50
\pinlabel $b$ at 64 50
\pinlabel $c$ at 79 30
\pinlabel $d$ at 78 6
\pinlabel $g$ at 67 81
\pinlabel $h$ at 96 37
\endlabellist
\centering
\includegraphics[width=10cm]{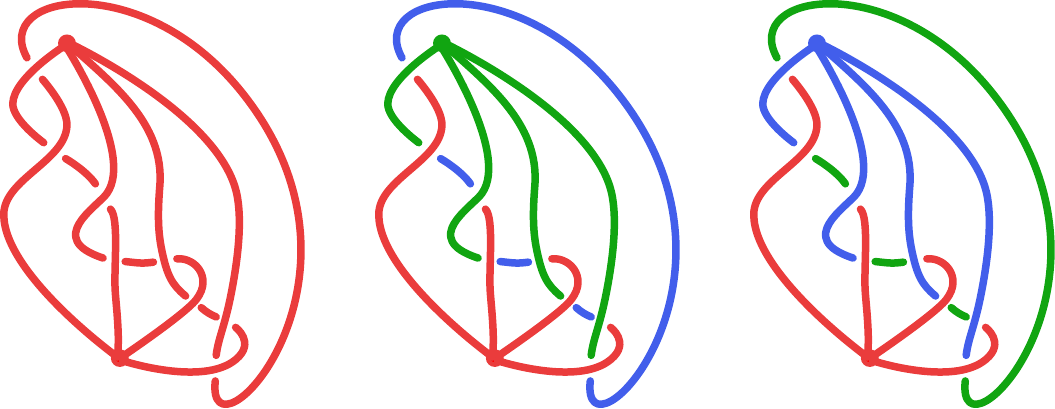}
\caption{For a Suzuki $\Theta_4$-graph, we see that the number of colorings where the arcs next to a particular vertex are labeled by a fixed element of the dihedral quandle of order three is three. (Here we use colors to represent the labelings by the quandle.) Thus by permuting the colors, we have a total of nine colorings. The labels in the leftmost diagram are used in the proof of \autoref{prop:newcolor}. \label{fig:vertexred}}
\end{figure}

\begin{proposition}
    The number of generalized $3$-colorings of the Suzuki $\Theta_4$ graph is $81=3^4$. \label{prop:newcolor}
\end{proposition}

\begin{proof}
    We get the following chains of equations, referencing the labels in the leftmost diagram of \autoref{fig:vertexred}.
\begin{align*}
y-z+t-w&=0 \pmod 3 & a-b+c-d &= 0 \pmod 3 \\
2y-x&=a \pmod 3 & 2a-y &= f \pmod 3 \\
2z-f &= b \pmod 3 & 2b-z&=g \pmod 3 \\
2t-g&= c \pmod 3 & 2c-t &= h \pmod 3 \\ 
2w-h &= d \pmod 3 & 2d-w &= x \pmod 3
\end{align*}
    These equations can be transformed to the following.
\begin{align*}
    a &= 2y - x\pmod{3} & b &= 2z - x\pmod{3} \\
    c &= 2t - x\pmod{3} & d &= 2w - x\pmod{3} \\
    f &= g = h = x 
\end{align*}
    In other words, any element of $\{0,1,2\}$ is a valid choice for $x,y,z,$ and $t$, giving $3^4$ total number of solutions. 
\end{proof}

This means that any bridge position in which one trivial tangle is made only of intervals requires at least four arcs in the trivial tangle.

\begin{corollary} \label{cor:Suzuki2}
   Let $G_n$ be the $\Theta_4$-graph obtained by taking the vertex sum of $n$ copies of Suzuki's graph. Then the bridge index $\beta(G_n)$ is $n+2$.
\end{corollary}

\begin{proof}
    By \autoref{cor:Suzuki1}, we have that the number of connected components in each tangle cannot be less than $n+1,$ but we have not ruled out the possibility that $\beta(G)= n+1.$ For that to happen, we must have that one of the trivial tangles is made only of intervals. However, \autoref{prop:newcolor} says that $G_1$ needs at least four arcs in a trivial tangle. Furthermore, when a vertex sum is performed, we can ensure that the colors at the strands where the sum happens have the same matching colors. For any fixed color at such a strand of one summand, there are three possible colored diagrams. Therefore, the number of generalized 3-coloring of $G_2=G_1\#_4 G_1$ is $3^4\cdot 3=3^5$. 
    Continuing with this logic, we conclude that the number of generalized 3-colorings of $G_n$ is $3^{n+3},$ implying that we need at least $n+3$ arcs of degree 2. So the (weighted) bridge index $\beta$ using only arcs in one trivial tangle is at least $n+3,$ which is strictly higher than bridge positions where vertices of the graphs are on different sides of the bridge splitting.
\end{proof}
\end{example}

\begin{example}[Vertex sums via clasp moves] \label{ex:vertex1} 
The clasp move \cite{simon1990minimally} (see \cite[Fig.~3]{jang2016new}), is a well-known move which preserves almost unknottedness of $\Theta$-graphs. We will use this to build a family of graphs as follows. First construct an almost unknotted $\Theta$-graph by performing the clasp move on a slightly isotoped trivial $\Theta$-graph as in \autoref{fig:ex_loop1} (left). Then let $G_n$ be the almost unknotted $\Theta$-graph obtained by repeating this procedure several times on the same graph, such that the clasps are ``stacked'' on top of each other, as indicated in \autoref{fig:ex_loop1} (right). Note that the resulting graph can also be obtained by taking the vertex sum of $n$ copies of the original graph. In this sense, we can view these vertex sums as coming from clasp moves.

We omit the proof of the following corollary, as it follows using either a very similar argument as for \autoref{cor:vertex}, or using the fact that this example can be seen as being obtained from vertex sums of the following, plus results in \cite{taylor2018additive} and \cite{taylor2021tunnel} about additivity of bridge index under vertex sum. 

\begin{corollary}
    Let $G_n$ be the almost unknotted $\Theta$-graph obtained by taking the vertex sum of $n$ copies of the graph obtained from the clasp move in \autoref{fig:ex_loop1} (left) on top of each other, as in \autoref{fig:ex_loop1} (right). 
    Then the unweighted bridge index $\widehat{\beta}(G_n)$ equals $2n+1$, and the bridge index $\beta(G_n)$ equals $\tfrac{4n+3}{2}$.
\end{corollary}

\begin{figure}[ht!]
\centering
\includegraphics[width=12cm]{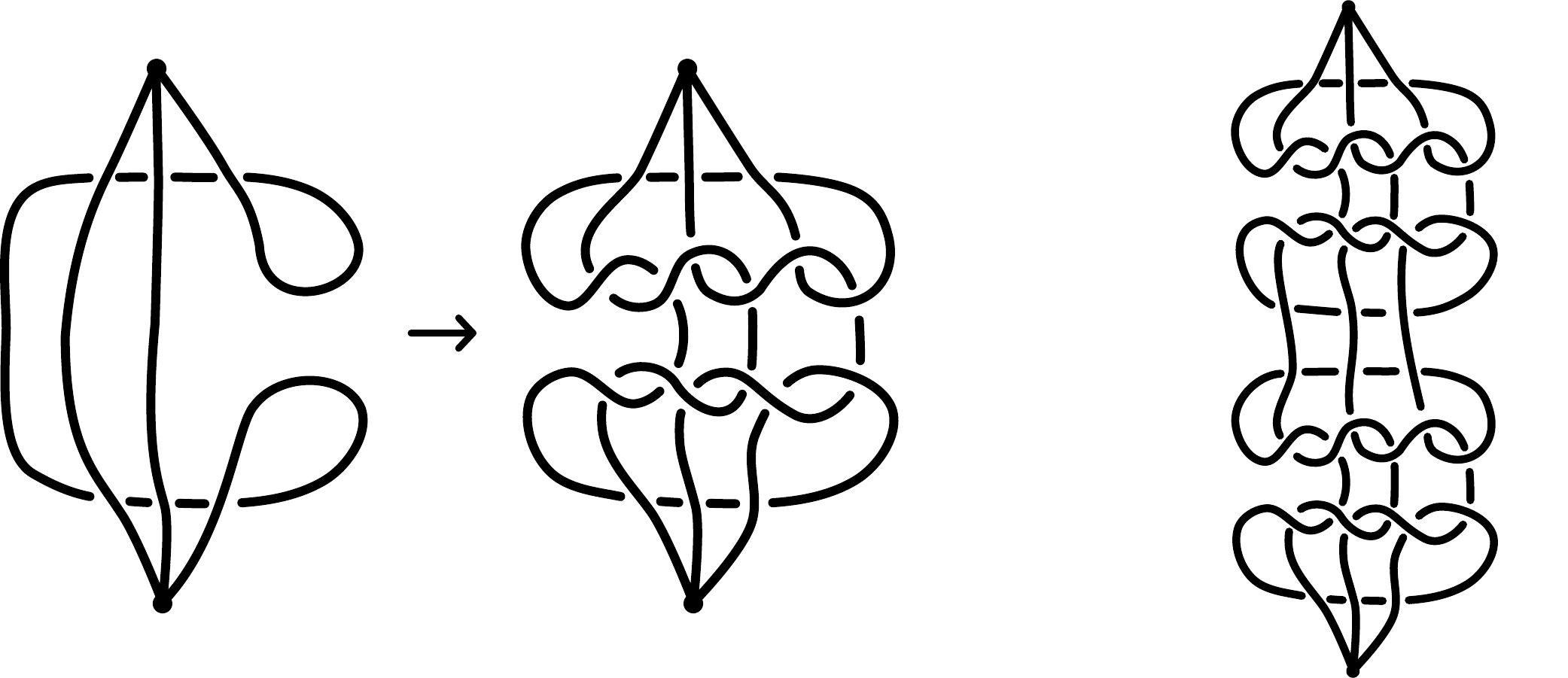}
\caption{Performing the clasp move to construct a new almost unknotted $\Theta$-graph (left). A $\Theta$-graph constructed by taking the vertex sum of copies of the graph obtained from the clasp move,  or equivalently, performing the clasp move multiple times (right). \label{fig:ex_loop1}}
\end{figure}

This family $G_n$ fits into a more general family of almost unknotted $\Theta_m$-graphs, constructed analogously, but we expect the bridge index to not change as $m$ increases. 
\end{example}

\begin{example}[More vertex sums via clasp moves] \label{ex:vertex2}
We can construct a somewhat simpler family of almost unknotted $\Theta$-graphs by performing ``half'' of the clasp move in the initial graph, as in \autoref{fig:ex_loop2} (left). Note that this graph fits into the family of Kinoshita-Wolcott graphs \cite{kinoshita1958alexander,kinoshita1972elementary,wolcott1987knotting} (see \cite[Fig.~1]{jang2016new}), and could be seen as being obtained from this perspective. Let $G_n$ be the family obtained from taking $n$ vertex sums of this graph, as indicated in \autoref{fig:ex_loop2} (right). Similarly, the graph obtained in \autoref{fig:ex_loop1} (left) could be seen as being obtained by taking the vertex sum of the Kinoshita graph and its mirror.

\begin{figure}[ht!]
\includegraphics[width=8cm]{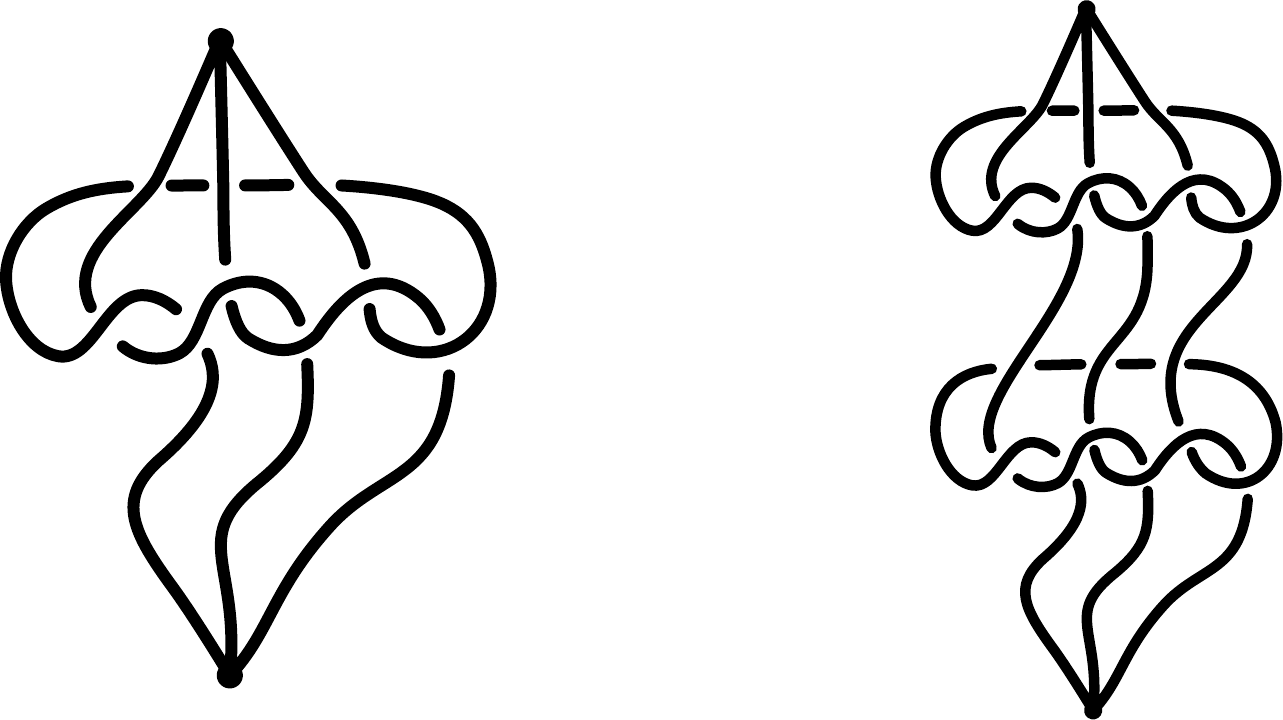}
\centering
\caption{The result of performing ``half'' of the clasp move to construct a new almost unknotted $\Theta$-graph (left). A $\Theta$-graph constructed by taking the vertex sum of copies of the graph on the left, or equivalently, performing the ``half'' clasp move multiple times (right). \label{fig:ex_loop2}}
\end{figure}

In \autoref{fig:ex_loop2color} we give four consistent labelings of $G_1$ by the homogeneous Alexander quandle of order four from \autoref{sec:quandle}. 
Observe that by permuting labels, the total number of colorings of $G_1$ is 16. By \autoref{prop:homogeneous}, the number of colorings of $G_n$ is $4^{n+1}$, and by a combination of the lower bound coming from \autoref{prop:color} and the upper bound coming from our algorithm, we see that the unweighted bridge index is precisely $n+1$.

\begin{figure}[ht!]
\labellist
\small\hair 2pt

\pinlabel $1$ at 90 10
\pinlabel $2$ at 25 144
\pinlabel $0$ at 70 144
\pinlabel $3$ at 189 215
\pinlabel $0$ at 89 225
\pinlabel $3$ at 129 225
\pinlabel $2$ at 189 255
\pinlabel $1$ at 80 290

\pinlabel $1$ at 360 10
\pinlabel $3$ at 295 144
\pinlabel $2$ at 340 144
\pinlabel $0$ at 459 215
\pinlabel $2$ at 359 225
\pinlabel $0$ at 399 225
\pinlabel $3$ at 459 255
\pinlabel $1$ at 350 290

\pinlabel $1$ at 630 10
\pinlabel $0$ at 565 144
\pinlabel $3$ at 610 144
\pinlabel $2$ at 729 215
\pinlabel $3$ at 629 225
\pinlabel $2$ at 669 225
\pinlabel $0$ at 729 255
\pinlabel $1$ at 620 290

\pinlabel $1$ at 900 10
\pinlabel $1$ at 835 144
\pinlabel $1$ at 880 144
\pinlabel $1$ at 999 215
\pinlabel $1$ at 899 225
\pinlabel $1$ at 939 225
\pinlabel $1$ at 999 255
\pinlabel $1$ at 890 290

\endlabellist
\includegraphics[width=13cm]{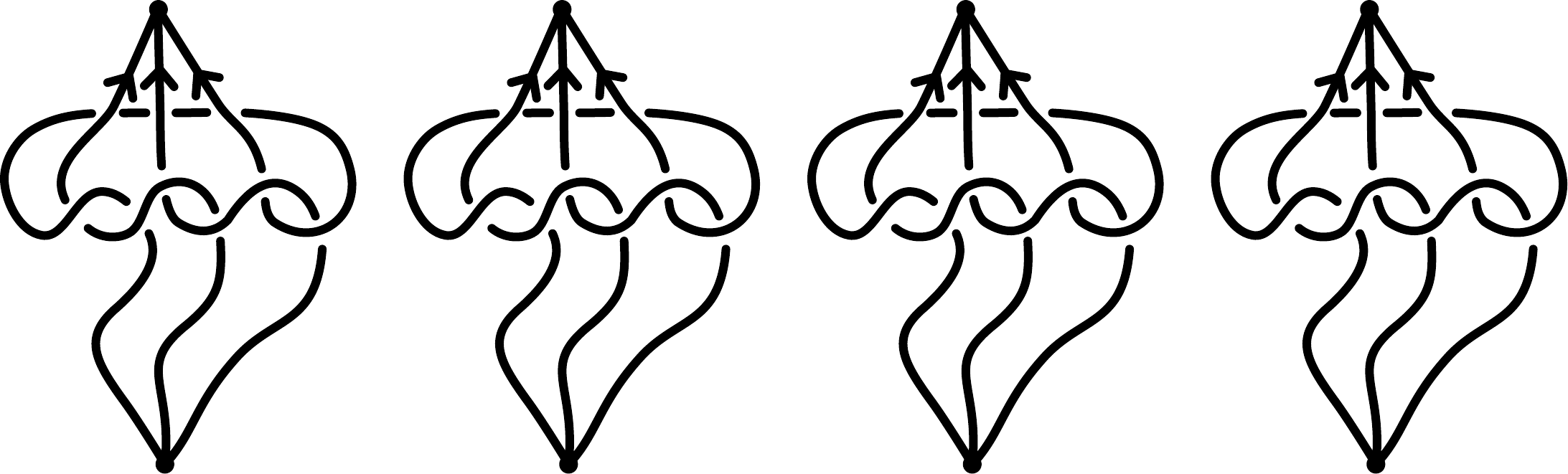}
\centering
\caption{Four consistent labelings of the graph in \autoref{fig:ex_loop2} (left) by the Alexander quandle of order four. Thus by permuting the labels, we have a total of sixteen colorings. \label{fig:ex_loop2color}}
\end{figure}

To determine the (weighted) bridge index, we consider possibilities for the structure of the trivial tangles that fit in a bridge splitting of a $\Theta$-graph with a prescribed number $k$ of unknotted tree components in the upper tangle: there are four cases (see \autoref{fig:poss} for $k=4$). 
Note that the configurations on the right column of \autoref{fig:poss} will never realize the bridge index, so we only need to consider the left column configurations. That is, the bridge index is either $\frac{2n+3}{2}$ or $\frac{2n+2}{2}$. However, the latter cannot happen because the lower trivial tangle is forced to have $n$ connected components, contradicting our lower bound of $n+1$ from the quandle coloring.

\begin{figure}[ht!]
\includegraphics[width=11cm]{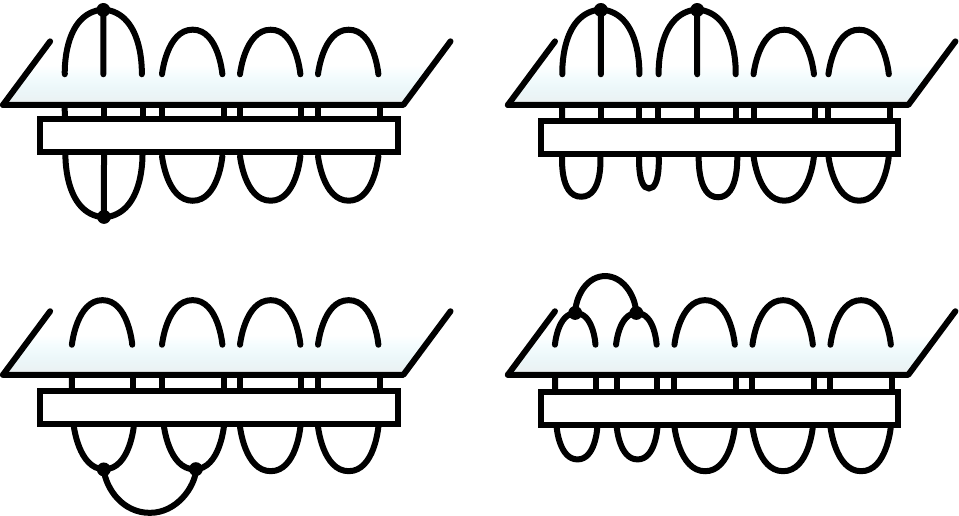}
\centering
\caption{Different possibilities for trivial tangles in a bridge splitting of a $\Theta$-graph, where the number of components for the tangle above is four. \label{fig:poss}}
\end{figure}

\begin{corollary} \label{cor:vertex}
    Let $G_n$ be the almost unknotted $\Theta$-graph obtained by taking the vertex sum of $n$ copies of the graph obtained from the clasp move in \autoref{fig:ex_loop2} (left) on top of each other, as in \autoref{fig:ex_loop2} (right). 
    Then the unweighted bridge index $\widehat{\beta}(G_n)$ equals $n+1$, and the bridge index $\beta(G_n)$ equals $\tfrac{2n+3}{2}$.
\end{corollary}
\end{example}

Our procedure for bounding the bridge index from below in this example yields the same lower bound previously achieved in \cite{taylor2021tunnel}. However, their bounds (in Theorems 7.4 and 7.5) include hypotheses such that their results can only be applied to the present example (and by association, \autoref{ex:Kinoshita} and \autoref{ex:vertex1}) in this paper.

\subsection{Eulerian spatial graphs from links}
These examples are related to constructions in \cite{flapan2017ravels}, where the authors were interested in nontriviality of the spatial graphs, but here we pay attention to the bridge index.

\begin{example}[Bouquet graphs]
First we consider the family of almost unknotted bouquet graphs $G_n$ shown in \autoref{fig:mont}. Using the terminology in \cite{flapan2017ravels}, the graph $G_n$ is a vertex sum of a Montesinos tangle, where no rational subtangle has the $\infty$ parity. This means that the arc in the diagram connected to the northwest endpoint of each tangle has the other endpoint in the southeast corner. By Theorem 3.3 of \cite{flapan2017ravels}, the spatial graph is ravel, which means that the spatial graph itself is nontrivial but every cycle is unknotted. Since a bouquet graph has the two cycles as proper subgraphs, being ravel is the same as being almost unknotted.

We claim that the bridge index of $G_n$ is at most $n+1$. The standard diagram of a Montesinos knot containing $n$ rational tangles has $n$ local maxima. To form a bouquet graph, we perform a vertex sum, which fuses a local maximum with a local minimum (see \autoref{fig:fuse}). This gives an embedding with $n-1$ local maxima and a vertex of degree four on one side of the bridge sphere. On the other side of the bridge sphere, there are $n+1$ local minima. Thus, the number of intersections of $G_n$ with the bridge sphere is $2(n+1).$ Dividing by $2$ gives the claim.

To show that the bridge index is at least $n+1$, we will use \autoref{prop:reflection}. The length $n$ Montesinos links admits a rank $n$ Coxeter quotient \cite{blair2024coxeter}. Performing a vertex closure as indicated in the figure still yields a rank $n$ Coxeter quotient since the strands  where the closure takes place are sent to the same reflection. Therefore, each bridge splitting must have $n$ components in one trivial tangle. Since one component has weight 4 and the other $n-1$ components have weight $2$, the bridge index is $\frac{2(n-1)+4}{2} = n+1.$

\begin{corollary}
    The bouquet graph $G_n$ in \autoref{fig:mont} is almost unknotted and has bridge index $n+1.$
\end{corollary}

\begin{figure}[t!]
\includegraphics[width=9cm]{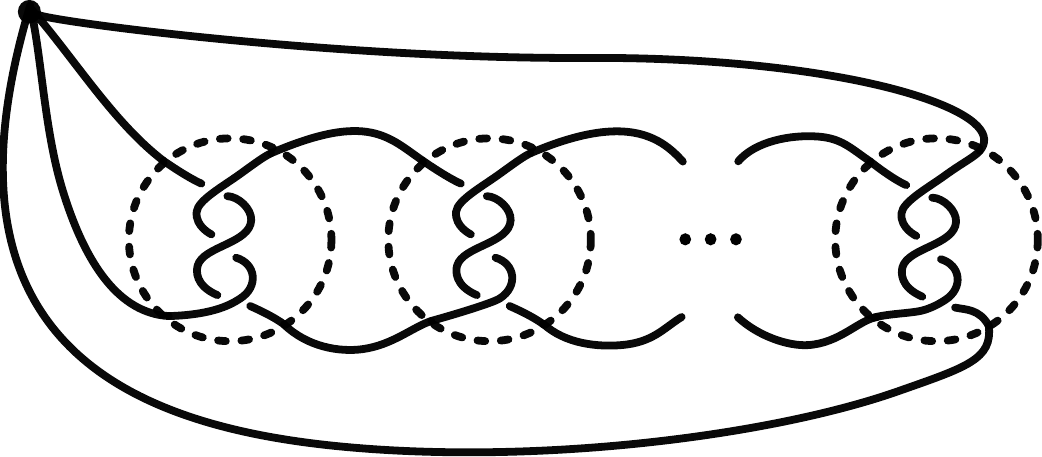}
\centering
\caption{A family of almost unknotted bouquet graphs with $n$ rational subtangles drawn as dashed circles. \label{fig:mont}}
\end{figure}
\end{example}

\begin{example}[More clasping] \label{ex:clasp}
We construct two spatial graphs by performing ``clasp-like'' moves in a different sense than \autoref{ex:vertex1}.

First, consider the two-component spatial graph shown in \autoref{fig:ex_clasp1} (left). Note that this graph is almost unknotted and has unweighted bridge index $3$ (which we will verify in \autoref{ex:Montesinos}). Further note that increasing the number of twists in various places, as long as each twist region remains an odd number of half-twists, will not change the linkedness, almost unknottedness, or bridge index of the graph.

Second, consider the $\Theta_4$-graph shown in \autoref{fig:ex_clasp1} (right). Note that this graph is \textit{not} almost unknotted, but it also has unweighted bridge index $3$. Further note that increasing the number of twists in various places, as long as the parity of the number of half-twists of each twist region is preserved, will not change the bridge index of the graph.

For both graphs, if we create a family of graphs from them by adding further pairs of edges that are clasped together in the same way as the two pairs in the given graphs, as shown in \autoref{fig:ex_clasp2} for \autoref{fig:ex_clasp1} (left), the upper bound for the bridge index coming from our algorithm increases accordingly, and we conjecture that this upper bound is in fact the bridge index. Note however that the almost unknottedness of the spatial graph in \autoref{fig:ex_clasp1} (left) will be destroyed once more edges are added.

\begin{fact} 
    Let $G_n$ be the $\Theta_{2n}$-graph obtained by increasing the number of edges of either of the graphs in \autoref{fig:ex_clasp1} as described above. 
    Then the unweighted bridge index $\widehat{\beta}(G_n)$ is less than or equal to $n+1$, and the bridge index $\beta(G_n)$ is less than or equal to $2n$.
\end{fact}

\begin{figure}[t!]
\includegraphics[width=9cm]{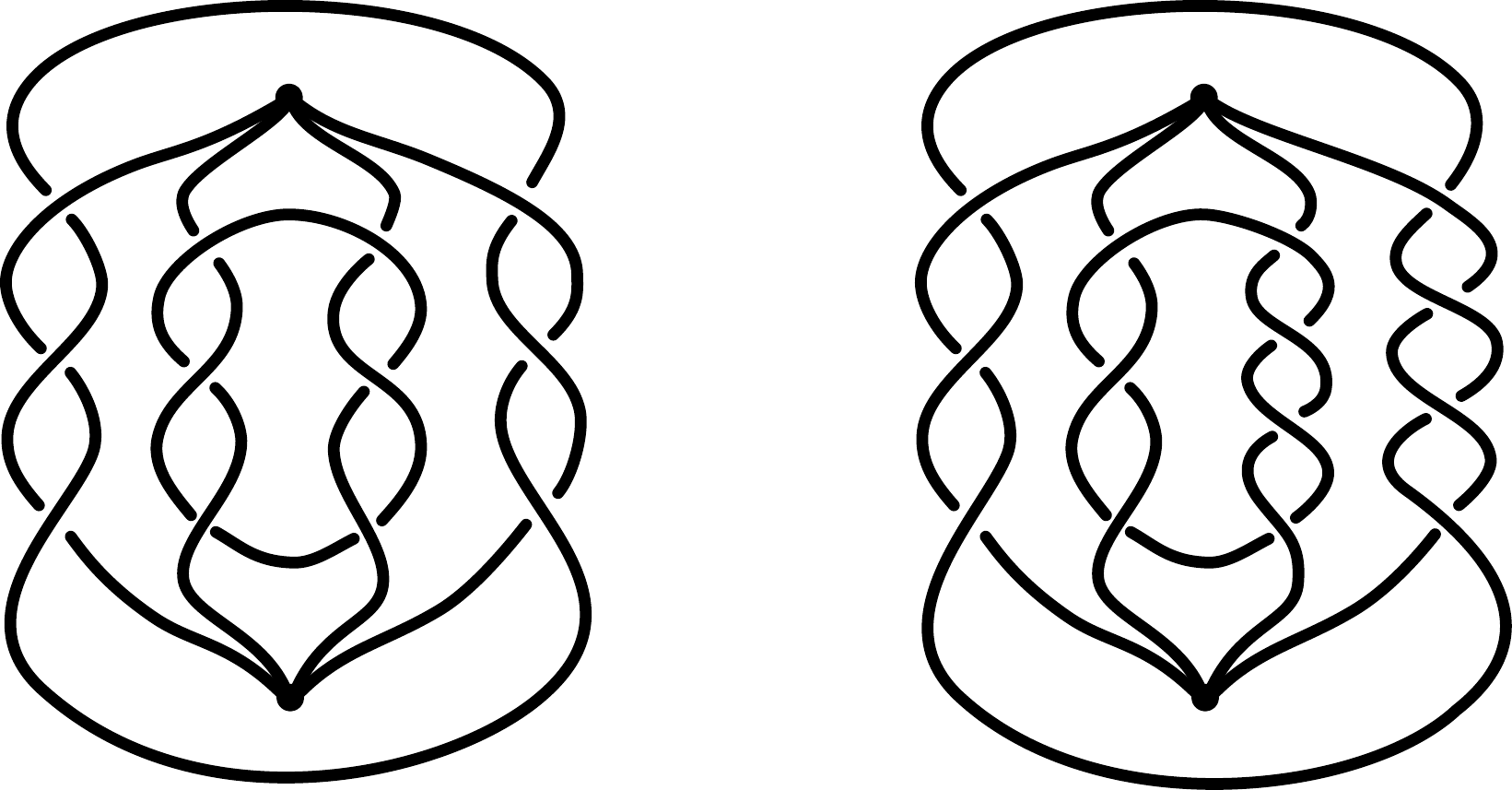}
\centering
\caption{Two more spatial graphs constructed from ``clasp-like'' moves. \label{fig:ex_clasp1}}
\end{figure}

\begin{figure}[t!]
\includegraphics[width=6cm]{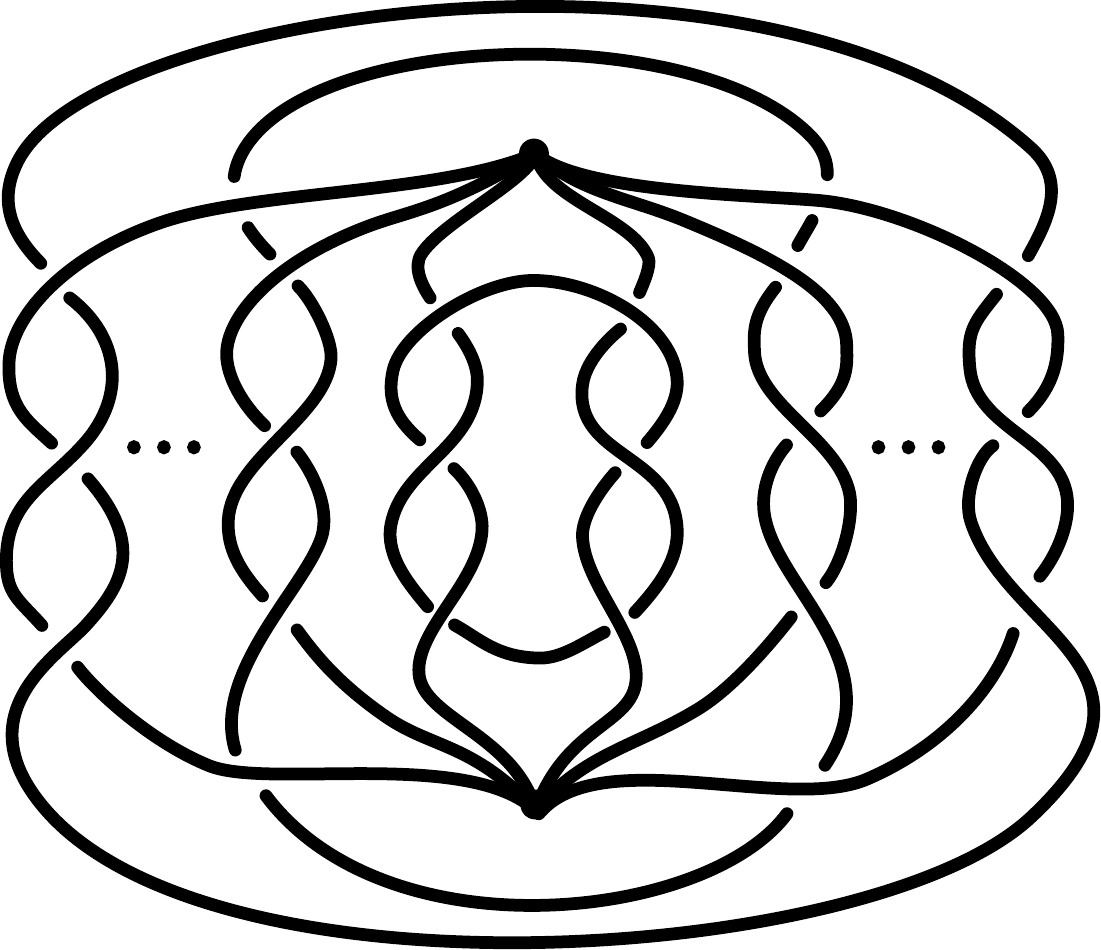}
\centering
\caption{A family of graphs created from \autoref{fig:ex_clasp1} (left) by adding further pairs of clasped edges. \label{fig:ex_clasp2}}
\end{figure}
\end{example}

Another way of viewing the construction in \autoref{ex:clasp} is by fusing pairs of maxima and minima of pretzel links, which leads us to a more general construction. Recall that a graph is \textit{Eulerian} if every vertex has even degree. There are numerous collections of knots and links that have an arbitrarily large number of quandle colorings \cite{przytycki20063} or admit Coxeter quotients with arbitrarily large reflection rank \cite{blair2024coxeter}. Therefore, the knots in these collections have arbitrarily large bridge index. Taking a knot and fusing pairs of local maxima or pairs of local minima together creates an Eulerian graph with degree four vertices (see \autoref{fig:fuse}). In various cases, we can obtain an almost unknotted graph this way. Furthermore, we inherit some lower bounds from the links. 

\begin{figure}[ht!]
\centering
\includegraphics[width=8cm]{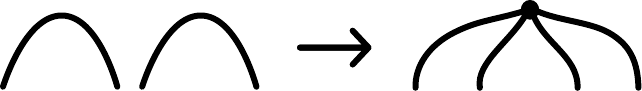}
\caption{Fusing two extrema to form a vertex of degree four. \label{fig:fuse}}
\end{figure}

Recall from \autoref{sec:rank} that we can use labelings by a Coxeter group to obtain lower bounds for bridge indices. Following \cite{blair2024coxeter}, we call a surjective group homomorphism  $\rho \colon \pi_1(S^3\backslash L)\twoheadrightarrow C$ from the fundamental group of a link to a Coxeter group $C$ a \textit{good quotient} if meridians are mapped to reflections. 

\begin{example}[Graphs from Montesinos links] \label{ex:Montesinos}
Consider the left picture of \autoref{fig:infinite}. Once rational tangles are added in the circles, this is a Montesinos link, and the authors of \cite{baader2021coxeter} have shown that it admits a good quotient. In particular, the meridional rank and the bridge index is $n$, where $n$ is the number of rational tangles inserted. Each maximum is associated to a reflection. The right of \autoref{fig:infinite} gives a bridge position of the spatial graph of interest with $n-1$ maxima, so the unweighted bridge index is at most $n-1$. Suppose that we have another bridge position for the graph with $b$ maxima. We can replace the vertices with the trivial $2$-tangles that recover the original Montesinos link. This is then a bridge position for the link with $b+1$ maxima. By \cite{baader2021coxeter}, we have $b+1\geq n,$ so $b\geq n-1$, the unweighted bridge index is equal to $n-1$, and the bridge index is equal to $n$.

Thus we can construct infinitely many graphs for which our technique can produce the exact
bridge index, where in this case, the graphs are not just coming from a family of vertex sums
as in previous examples. Of course, some choices will not give an almost unknotted graph,
but some choices, such as \autoref{fig:ex_clasp1} (left), will.

\begin{corollary} \label{cor:Montesinos}
    Let $G_n$ be the spatial graph obtained from fusing some of the maxima and minima of a Montesinos link with $n$ rational tangles, as depicted in \autoref{fig:infinite}. 
    Then the unweighted bridge index $\widehat{\beta}(G_n)$ equals $n-1$, and the bridge index $\beta(G_n)$ equals $n.$
\end{corollary}
    
\begin{figure}[ht!]
\centering
\includegraphics[width=13cm]{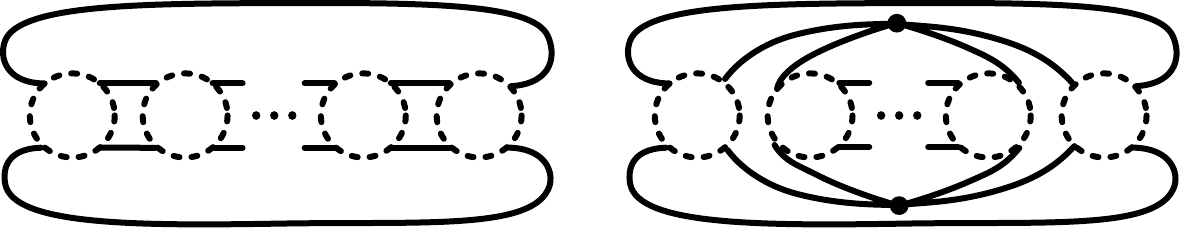}
\caption{A schematic to construct infinitely many graphs for which our technique can produce the exact bridge index. The dashed circles represent different rational tangles. \label{fig:infinite}}
\end{figure}
\end{example}

We now prove a general result on lower bounds of bridge index from repeatedly replacing vertices with rational tangles. In particular, we can use this to recover \autoref{cor:Montesinos}, but it is also useful in other situations.

\begin{proposition}\label{prop:fusegroup}
Suppose that $v$ is a vertex of a spatial graph $G$ of even degree $d$. If we replace $v$ with a $\tfrac{d}{2}$-string trivial tangle to get a spatial graph $G'$ with one fewer vertex, then there is a surjection from the fundamental group of the exterior of $G$ to that of $G'$ taking meridians to
meridians.
\end{proposition}

\begin{proof}
The boundary of the exterior of $G$ is a surface $F$. The tangle replacement that replaces a small neighborhood of a vertex with a $\tfrac{d}{2}$-trivial tangle can be thought of as attaching $\tfrac{d}{2}-1$ three-dimensional 2-handles with attaching regions on $F$; see \autoref{fig:2handles}. By Van Kampen's theorem, 2-handle additions only add relations to the fundamental group. In other words, the fundamental group of the exterior of $G'$ is a quotient of the group for $G$.
\end{proof}

\begin{figure}[ht!]
\centering
\includegraphics[width=8.5cm]{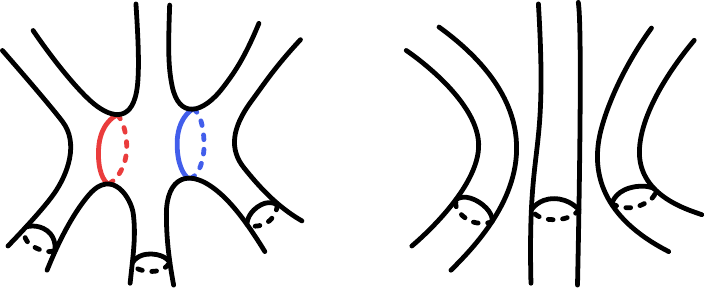}
\caption{Replacing a vertex with rational tangles has the effect of attaching some number of $2$-handles. In this particular figure, the vertex has degree six, and we attach $2$-handles along the red and blue curves to get a $3$-string tangle. \label{fig:2handles}}
\end{figure}

\subsection{Arbitrarily large bridge index} \label{sec:large}

The goal of this subsection is to prove \autoref{thm:index}. We begin with a series of lemmas. 

\begin{lemma}
    A diagram of a bouquet graph of $n$ petals with Wirtinger number $n$ coming from one seed is planar.
\end{lemma}

\begin{proof}
    The unique vertex $v$ of the graph has $2n$ edges emerging from it. If one seed extends to the entire diagram, then there is an embedding where one trivial tangle is a connected tree containing $v$. The other trivial tangle is made up of $n$ arcs containing local minima. A diagram representing this graph may contain crossings as we move away from $v$ traversing along edges. However, each crossing can be removed by the Reidemeister move (see move R6 from \cite{khandhawit2024planar}, for example) that allows twisting near vertices until no crossings are left.
\end{proof}

As mentioned in \cite{jang2016new,taylor2021abstractly}, it is not understood in general how clasping moves affect spatial graphs besides preserving almost unknottedness. In this paper, we show that by picking the clasping carefully in such a way that the quandle colorings are compatible, we can drive up the bridge index to be arbitrarily large while still being almost unknotted. This method can be modified to be compatible for colorings for even more general algebraic structures such as biquandles, but we just focus on the dihedral quandle for now. 

Recall that a knot or spatial graph which admits a nontrivial quandle coloring by the dihedral quandle of order three $R_3$ is called \textit{tricolorable}. For convenience, by a \textit{Type I} replacement, we mean the replacement $T\mapsto T'$ on three strands with the same color as shown on the right in \autoref{fig:organizedclasp}. By a \textit{Type II} replacement, we mean the replacement $T\mapsto T'$ where the colorings at $\partial T'$ are as in \autoref{fig:works} (the colors are interchangeable, but the point is that the upper right and lower right endpoints of the tangle get one color and the remaining endpoints get another color).

\begin{remark}
    The conditions on the dihedral colorings at $\partial T'$ are important. The reader can verify that for some other choice of colorings at $\partial T'$, the number of colorings is strictly less than claimed in the next lemma.
\end{remark}

\begin{figure}[ht!]
\includegraphics[width=8.5cm]{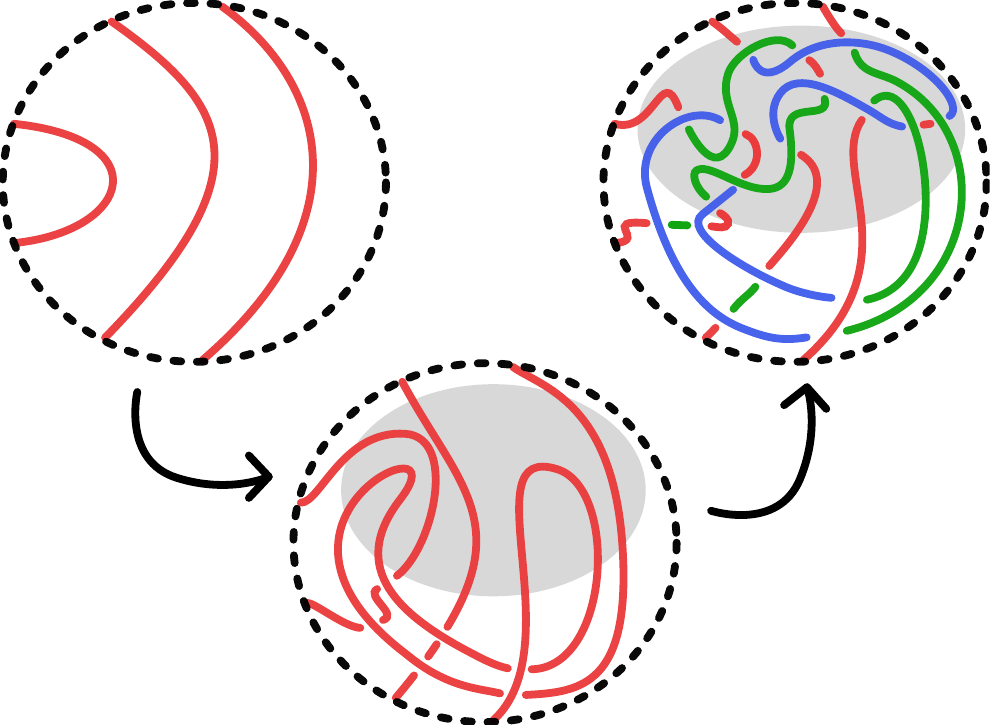}
\centering
\caption{A clasping move that increases the bridge index by one. \label{fig:organizedclasp}}
\end{figure}

\begin{figure}[ht!]
\includegraphics[width=11cm]{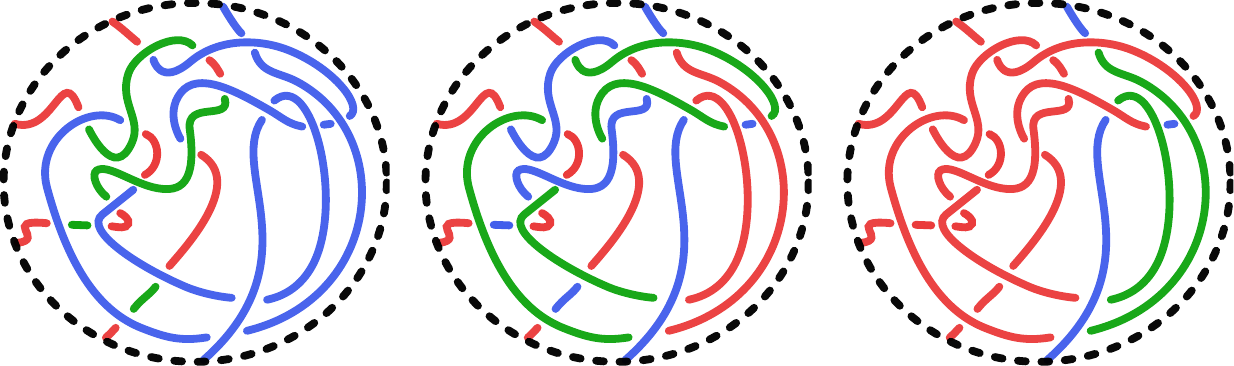}
\centering
\caption{Three tricolorings of $T'$ where there are two colors at the tangle boundaries. \label{fig:works}}
\end{figure}

\begin{lemma} \label{lem:colorincreases}
    Suppose $D$ is a diagram of a tricolorable almost unknotted graph with the number of colorings $Col_{R_3}(D)$. After one instance of the tangle replacement of Type I or Type II, we have that $D'$ is still almost unknotted and $Col_{R_3}(D')=3\cdot Col_{R_3}(D)$.
\end{lemma}

\begin{proof}
    To see that $D'$ represents an almost unknotted graph, we view the tangle replacement in two stages as shown in \autoref{fig:organizedclasp}. The first stage comprises entirely of Reidemeister moves which do not change the spatial graph type. The second stage is the clasping move, which is known to preserve almost unknottedness.

    The colorings shown in \autoref{fig:organizedclasp} and \autoref{fig:works} also demonstrate the claim on the formula for the quandle coloring number of $D'$. To be more precise, we name all the tricolored diagrams of $D$ as $D_1,D_2,\ldots, D_{Col_{R_3}(D)}$. After the replacement, we see that each colored diagram $D_j$ corresponds to 3 colorings of $D'$. 
    
    Namely, for the Type I replacement, the three colorings are the version where all strands of $T'$ are colored the same and two additional non-monochromatic colors for $T'$ (in \autoref{fig:organizedclasp}, notice that green and blue can be swapped). For the Type II replacement, the three colorings for each $D_j$ are shown in \autoref{fig:works}. In conclusion, the total number of colorings of $D'$ is $3\cdot Col_{R_3}(D)$ as claimed.
\end{proof}

The next lemma is straightforward. We use the same notation for the tangle replacement $T\mapsto T'$. In \autoref{fig:claspwirtseq}, the opaque gray oval is shown to mark the spot where tangle replacements can be performed to drive up the bridge index.

\begin{lemma}
      Suppose that $D\backslash T=D\backslash T'$ can be partially colored with $k$ seeds so that four strands which intersect $\partial T$ as shown in \autoref{fig:claspwirtseq} (marked with $x$'s) receive colors. After one instance of the tangle replacement in \autoref{fig:organizedclasp} to get a diagram $D',$ we have that $D'$ is still almost unknotted and has $k+1$ seeds, where the new seed has weight $2$.\label{lem:wirtseq}
\end{lemma}

\begin{proof}
    Maximally extend the $k$ seeds for $D\backslash T'$. By assumption, the strands that intersect $\partial T=\partial T'$ labeled with the letter $x$ are colored by the end of this process. Parts of these strands are contained in $T',$ but the current choice of seeds is not enough to color the entire $T'.$ An additional seed of weight $2$ (colored blue in \autoref{fig:claspwirtseq}) can be used to complete the coloring to the entire tangle $T'$.
\end{proof}

\begin{figure}[ht!]
\labellist
\small\hair 2pt
\pinlabel  $x$ at 42 385
\pinlabel $x$ at 93 375
\pinlabel $x$ at -6 438
\pinlabel $x$ at -7 497
\endlabellist
\includegraphics[width=11cm]{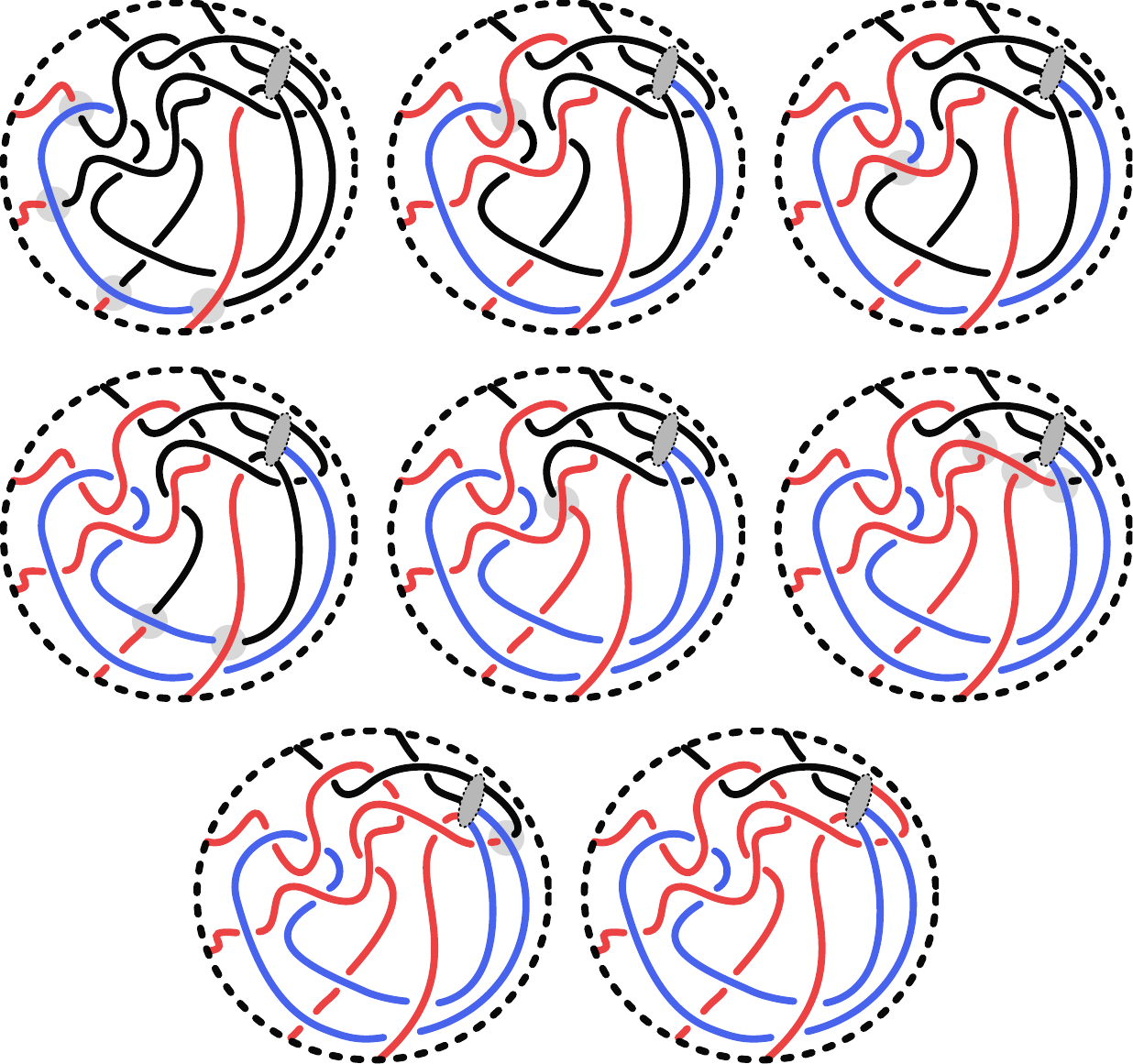}
\centering
\caption{Suppose that the diagram before the tangle replacement has Wirtinger number $k$ satisfying \autoref{lem:wirtseq}. Then, the diagram after the tangle replacement has Wirtinger number $k+1.$ Note that a Reidemeister 2 move has been performed to make the two seeds extend. We have put in an opaque gray oval to indicate the spot to iteratively insert $T'$ to drive the bridge index to be as large as needed. The sequence of diagrams should be viewed left to right, top to bottom. \label{fig:claspwirtseq}}
\end{figure}

We are now ready to prove \autoref{thm:index}. We will repeatedly discuss $3$-string tangle replacements, and so it will be convenient to name endpoints of the tangle NE, NC, NW, SE, SC, and SW as depicted in \autoref{fig:twisttangle}. Throughout, we are using \autoref{thm:main} for upper bounds and \autoref{prop:color} for lower bounds. Note that while \autoref{prop:color} is a lower bound in terms of $\widehat{\beta}$ (the unweighted version), it still works for bouquet graphs as there is no ambiguity of where to place the vertex: the vertex has to be in one of the trivial tangles. In contrast, when one has a $\theta_n$-graph, for example, the choice of whether to put the vertices on the same side will affect the weighted versus the unweighted version.

\begin{figure}[ht!]
\labellist
\small\hair 2pt
\pinlabel {$NW$}  at 460 173
\pinlabel {$NC$}  at 505 173
\pinlabel {$NE$}  at 600 173
\pinlabel {$SW$}  at 460 -13
\pinlabel {$SC$}  at 505 -13
\pinlabel {$SE$}  at 600 -13
\endlabellist
\vspace{1em}
\includegraphics[width=12cm]{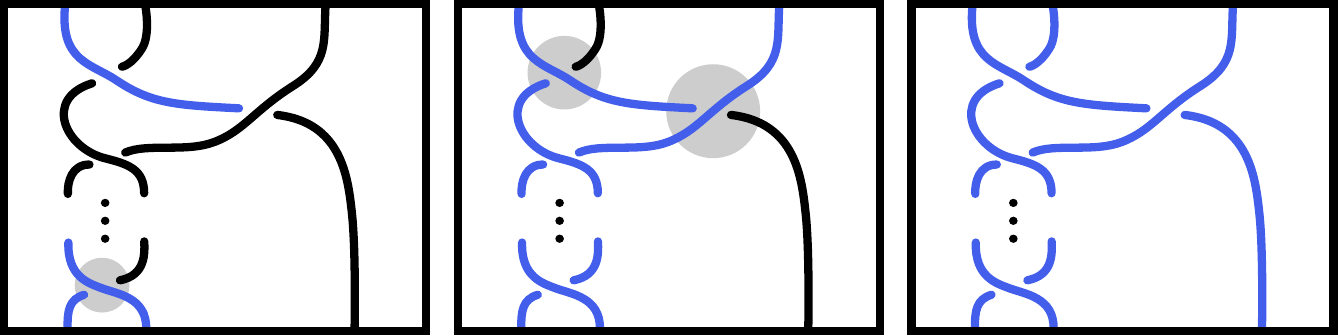}
\vspace{1em}
\centering
\caption{Extending the seed colorings through each twist knot. \label{fig:twisttangle}}
\end{figure}

\index*

\begin{proof}
Observe that the bouquet graph of $n$ petals has Euler characteristic $-(n-1)$. Thus, by taking $n$ to infinity, we obtain all possible Euler characteristics for interesting graph types that admit almost unknotted embeddings. 

Let $K_1,K_2,\ldots, K_n$ be tricolorable twist knots. In particular, each $K_i$ is a closure of the tangle in \autoref{fig:twisttangle} and each of these admit an embedding with two local maxima and two local minima. We fuse the lower left local minima to become the bouquet graph vertex shown on the left of \autoref{fig:specialbouq}. We also splice the lower right local minima in the manner shown in the figure, and call the resulting graph $B_n$.

\begin{figure}[ht!]
\includegraphics[width=11cm]{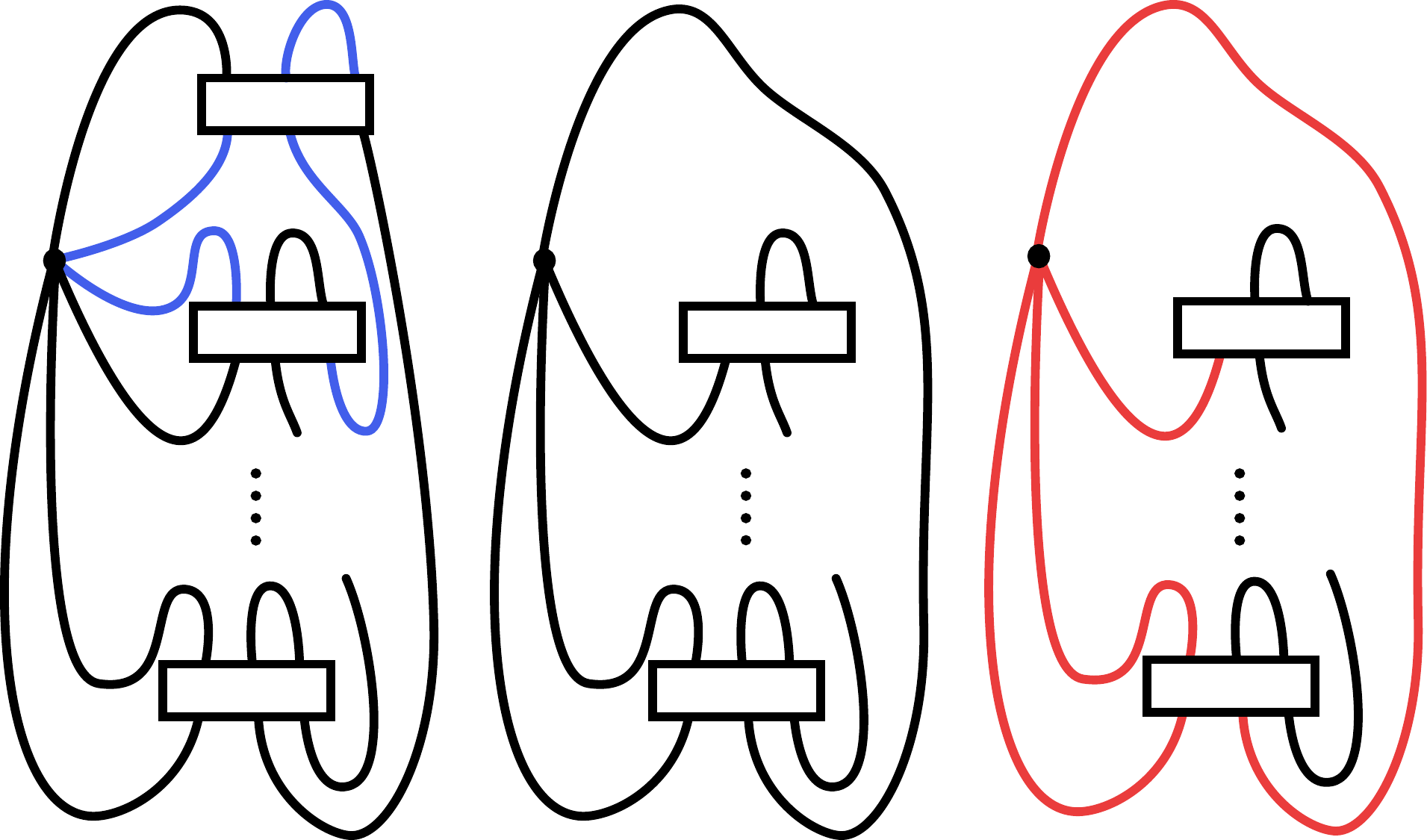}
\centering
\caption{An almost unknotted bouquet graph of $n$ petals which we will perform clasping to. The blue strand corresponds to a cycle. There are $n$ rectangular boxes each corresponding to a rational tangle defining $K_n$. \label{fig:specialbouq}}
\end{figure}

We claim that removing any edge $e$ results in a planar graph. We show this by applying the main theorem of our paper: the Wirtinger number of $B_n\backslash e$ is $n-1$. When a cycle is removed, we get a diagram as shown in the middle of \autoref{fig:specialbouq}. The red strand in the figure is the only seed $s$ of weight $2n$ we need to generate the coloring of the entire diagram. Notice that $s$ propagates to color three endpoints of each $3$-braid represented by a rectangular box. \autoref{fig:twisttangle} tells us how to extend the coloring coming from $s$ through the entire 3-braid. After that, we can repeat the process to color all 3-braid boxes in the diagram.

The graph $B_n$ has bridge index precisely $n+1.$ To see this, note that the green and blue strands in a diagram $D$ as shown in \autoref{fig:extend} function as the two seeds: one with weight $2n$ and one with weight $2$. The lower bound for this claim on $B_n$ comes from our assumption on tricolorable twist knots. As before, handling the lower bound, treat the colors at the seed to be quandle labels. This means that the quandle labels propagate consistently so that the NW and SW strands receive the same color, say green. Furthermore, the remaining SC and SW strands receive the same color, say blue.

\begin{figure}[ht!]
\includegraphics[width=7.5cm]{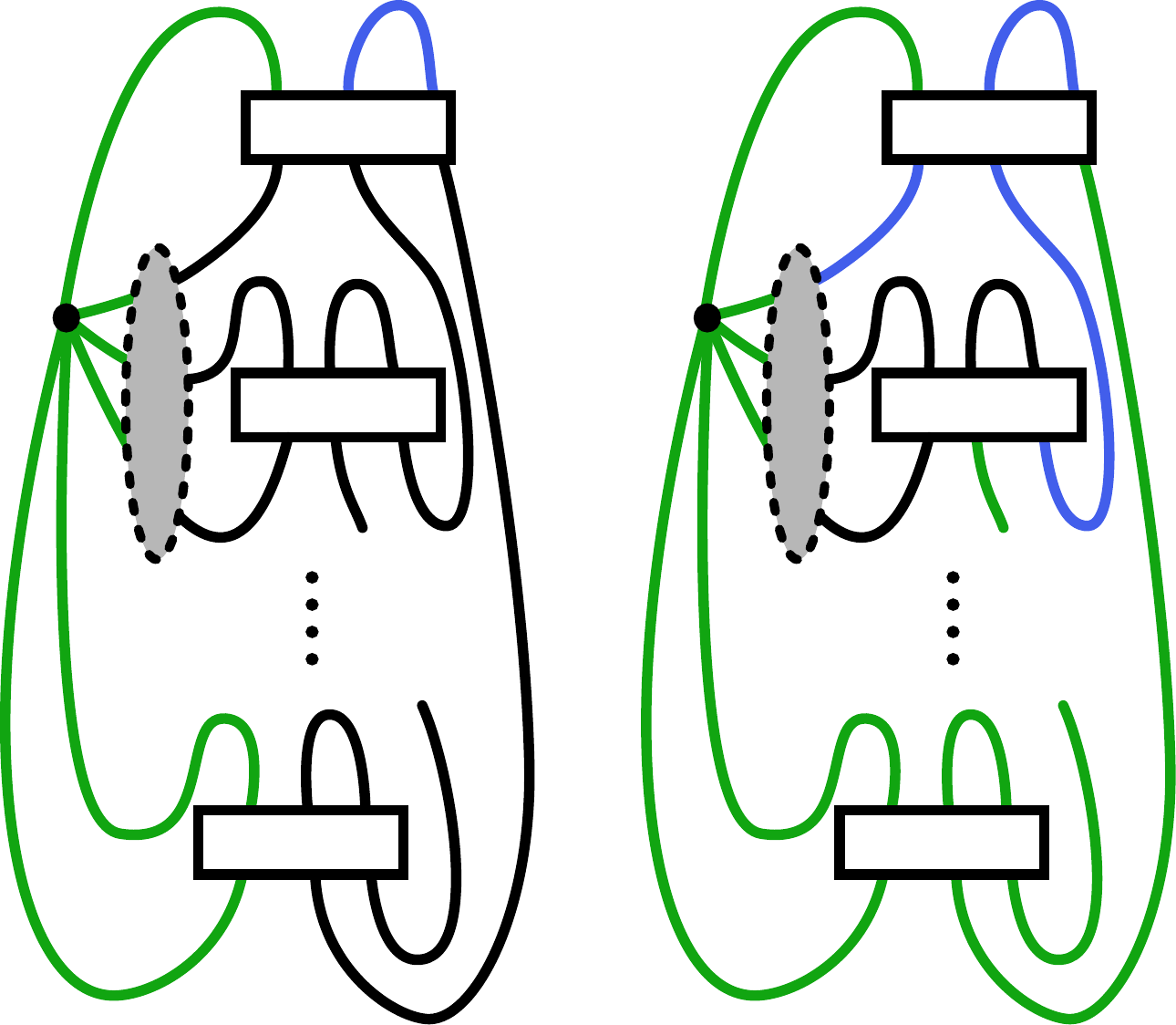}
\centering
\caption{The diagram $D\backslash T$ can be colored with 2 seeds (one with weight $2$ and one with weight $2n$) so that four out of the six strands that intersect $\partial T$ get colored. \label{fig:extend}}
\end{figure}

Now, to see that the bridge index can be arbitrarily large for a fixed $n,$ we start with $B_n$ and perform a tangle replacement $T\rightarrow T'$ as shown in \autoref{fig:extend}, and call the resulting graph $B_n^2$, represented by a diagram $D^2$. By \autoref{lem:colorincreases} and \autoref{lem:wirtseq}, $B^2_n$ has bridge index precisely $n+2$. For the upper bound, note that the original two seeds from $B_n$ suffice to extend the coloring of $D\backslash T'$ maximally to satisfy the hypothesis of \autoref{lem:wirtseq} (the NW, SW, SC, and SE strands are colored). Therefore, an additional seed of weight $2$ will suffice to generate the coloring for the entire diagram. For the lower bound, observe that this is a Type I replacement: the diagram $D$ is tricolorable, and we are performing a replacement near a vertex where all arcs are colored the same.

Finally, the process to construct further iterations $B_n^3,B_n^4,\ldots, B_n^k$ can be described as follows. Now that $B_n^2$ already contains an instance of $T'$, we look inside $T'$ to perform more tangle replacements sequentially at the opaque gray oval in \autoref{fig:claspwirtseq}. We remark again that each replacement satisfies \autoref{lem:colorincreases} and \autoref{lem:wirtseq}. For the upper bound, notice that from the first picture to the last of \autoref{fig:claspwirtseq}, we can extend the colorings at the strands marked $x$ to $B^2_n\backslash T'$ maximally until the NW, SW, SC, and SE strands are colored as \autoref{lem:wirtseq} asks for. For the lower bound, we are performing a Type II replacement, where the quandle colorings at the endpoints of $T'$ are as shown in \autoref{fig:works}.
\end{proof}

\sloppy
\printbibliography

$\quad$ \\
Sarah Blackwell \\{University of Virginia}\\
Email address: \texttt{\href{mailto:blackwell@virginia.edu}{blackwell@virginia.edu}}\\
URL: \url{https://seblackwell.com/}\\
$\quad$ \\
Puttipong Pongtanapaisan \\{Pitzer College}\\
Email address: \texttt{\href{mailto:puttip@pitzer.edu}{puttip@pitzer.edu}}\\
URL: \url{https://puttipongtanapaisan.wordpress.com/}\\
$\quad$ \\
Hanh Vo \\{Purdue University}\\
Email address: \texttt{\href{mailto:hanhmfa@gmail.com}{hanhmfa@gmail.com}}\\
URL: \url{https://hanhv.github.io/}\\

\end{document}